\documentclass[numbers,webpdf,imanum]{preprint-version}%

\usepackage{bm}
\graphicspath{{Fig/}}
\usepackage[normalem]{ulem}

\theoremstyle{thmstyletwo}%
\newtheorem{theorem}{Theorem}
\newtheorem{lemma}{Lemma}
\newtheorem{remark}{Remark}%

\numberwithin{equation}{section}

\newcommand{\R}{\mathbb{R}}

\begin{document}

\appnotes{Paper}

\firstpage{1}

\title[Quasi-interpolation for the Helmholtz--Hodge decomposition]{Quasi-interpolation for the Helmholtz--Hodge decomposition}
\author{Nicholas Fisher
\address{\orgdiv{Department of Mathematics and Statistics}, \orgname{Portland State University}, \orgaddress{\state{Portland, OR} \postcode{97201}, \country{USA}}}}
\authormark{N. Fisher et al.}
\author{Gregory Fasshauer
\address{\orgdiv{Department of Applied Mathematics and Statistics}, \orgname{Colorado School of Mines}, \orgaddress{\state{Golden, CO} \postcode{80401}, \country{USA}}}}

\corresp[*]{Corresponding author: \href{email:wenwugao528@163.com}{wenwugao528@163.com}}
\author{Wenwu Gao*
\address{\orgdiv{School of Big Data and Statistics}, \orgname{Anhui University}, \orgaddress{\state{Hefei}, \country{PR China}}}}



\abstract{The paper aims at proposing  an efficient and stable quasi-interpolation based method for  numerically computing the Helmholtz--Hodge decomposition of a   vector field.  To this end, we first explicitly construct a    matrix kernel in a general form from polyharmonic splines such that it includes divergence-free/curl-free/harmonic matrix kernels as special cases. Then we  apply the matrix kernel to vector decomposition via the convolution technique together with the  Helmholtz--Hodge decomposition.  More precisely, we show that if we  convolve a vector field with a scaled divergence-free (curl-free)  matrix kernel, then the resulting divergence-free (curl-free) convolution sequence converges to the corresponding  divergence-free (curl-free)  part of the   Helmholtz--Hodge decomposition of the field. Finally, by discretizing the convolution sequence via certain quadrature rule, we construct a  family of (divergence-free/curl-free) quasi-interpolants  for  the Helmholtz--Hodge decomposition (defined both in the whole space and over a bounded domain). Corresponding error estimates derived in the paper show that our quasi-interpolation based method yields convergent approximants to both the vector field and its Helmholtz--Hodge decomposition.}
\keywords{Divergence/curl-free/harmonic quasi-interpolation; Vector-valued function approximation; Helmholtz--Hodge decomposition; Leray projection; Polyharmonic splines.}


\maketitle

\section{Introduction}
The Helmholtz--Hodge decomposition has broad applications in applied mathematics, most notably in the study of computational fluid dynamics \cite{chorin, fsw, temam} and image processing \cite{pp}. In short, the Helmholtz--Hodge decomposition of a vector field $\bold f$ goes as $\bold f = \bold w + \nabla p$, where $\bold w$ is divergence-free ($\nabla \cdot \bold w = 0$) and $\nabla p$ is curl-free ($\nabla \times (\nabla p) = \bold 0$), though more general versions can be developed \cite{Schwarz}. Understandably, numerical methods for computing vector decompositions may vary based on setting and application. For example, in the context of computational fluid dynamics, one typically first solves a Poisson equation for the pressure term and then uses this to project the velocity term onto the space of divergence-free functions in a separate step \cite{gms}. On periodic domains, methods utilizing wavelets \cite{Deriaz} and meshless kernel methods \cite{Freeden}, \cite{fw2} have been put to use, while on bounded domains, radial basis function methods have been studied \cite{fw}, \cite{Wendland3}, \cite{Wendland5}, \cite{ZMRS0}. Many of the aforementioned techniques are based on divergence-free/curl-free interpolation methods, a problem which has been well studied  \cite{DoduandRabut}, \cite{Fuselier, Fuselier1}, \cite{Wendland1}, \cite{Wendland3},  \cite{VenelandBeatson},  \cite{Wendland5}. A common issue of these  interpolation based methods is numerical instability \cite{Drake}. To circumvent this issue, we adopt quasi-interpolation instead of interpolation for numerically computing vector decompositions motivated by fair properties and wide applications of quasi-interpolation  \cite{BeatsonPowell}, \cite{Buhmann, Buhmann1, BuhmannandDyn}, \cite{CharlesChuiandDiamond, CharlesChuiandDiamond1, CharlesChuiandLai}, \cite{GaoandWu4, Gaoetal}, \cite{Gaoandsun}, \cite{Lanzaraandmazya, Lanzaraandmazya1, Lanzaraandmazya2}, \cite{Rabut}, \cite{Schabackandwu}, \cite{ZMRS, GeneralizedWu}. 

The application of quasi-interpolants to vector-valued (or manifold-valued) data is well understood \cite{Amodei}, \cite{Atteia}, \cite{Benbourhim}, \cite{ChenandSuter}, \cite{Grohs, Grohs1}, but it was only until recently that divergence-free quasi-interpolants were introduced \cite{Gaoetal2}. The key idea is to first construct a scaled divergence-free matrix kernel such that it behaves like a delta distribution when it is convolved with a  divergence-free function and then discretize  the convolution sequence to get the final quasi-interpolant. However, since it is assumed that the data  must be sampled from a divergence-free function,  the resulting divergence-free quasi-interpolant can \textbf{not} be directly applied to reconstruct a general vector field $\bold f$. To overcome this pitfall, the paper  constructs a general quasi-interpolation scheme such that it can be used as an alternative for numerically computing the Helmholtz--Hodge decomposition of $\bold f$ and consequently reconstructing $\bold f$.

 We first construct a generalized matrix kernel from polyharmonic splines in an explicit form  (cf.~Equation \eqref{divergencefreekernl}) 
$$  {}_{\eta,\beta,\gamma}\bm{\Psi}_{\ell,k}=(-1)^{\ell}q_{d,\ell,k}(\widetilde{\Delta})[\eta\Delta {\bf I}_d - \beta\nabla \nabla^T](\nabla \nabla^T)^{\gamma}\phi_{\ell+\gamma+\max\{\eta,\beta\}}, \  \eta, \beta, \gamma \in\{0,1\}, \eta+ \beta+ \gamma\neq 0.$$
 Here $\ell$, $k$ are some positive integers satisfying $k\leq \ell$, $q_{d,\ell,k}$ is a polynomial of order $\ell+k$ in $d$ variables, $\widetilde{\Delta}$ is a difference operator, $\bold I_d$ is a $d\times d$ identity matrix, and $\phi_{\ell+m}$, $m = 1,2$, is a fundamental solution of the iterated Laplace operator $\Delta^{\ell+m}$ in $d$ variables in a distributional sense. The matrix kernel includes divergence-free, curl-free, and harmonic matrix kernels as special examples by choosing appropriate values of parameters $\eta, \beta, \gamma$.
 We also study  properties of ${}_{\eta,\beta,\gamma}\bm{\Psi}_{\ell,k}$ in both the spatial domain (see Lemma \ref{lem:convolutionerror}) and the Fourier domain (see Equations \eqref{fouriertrans}--\eqref{fouriertrans2}). 
 
 We then show that  the convolution of   $\bold f$ with a scaled version   of  the matrix kernel  approximates the divergence-free (curl-free) part of $\bold f$. More precisely, let $\bold f$ be a vector field having the Helmholtz--Hodge decomposition  $\bold f=\bold w+\nabla p$, where $\bold w$ and $\nabla p$ are   divergence-free part and curl-free  part of $\bold f$, respectively,   let 
 $${}_{\eta,\beta,\gamma}\bm{\Psi}_{\ell,k}^h(\bold x):=h^{-d}{}_{\eta,\beta,\gamma}\bm{\Psi}_{\ell,k}({\bold x}/h)$$ and define $${}_{\eta,\beta,\gamma}\bm{\Psi}_{\ell,k}^h*\bold f(\bold x):=\int_{\R^d}{}_{\eta,\beta,\gamma}\bm{\Psi}_{\ell,k}^h(\bold x-\bold t)\bold f(\bold t)d\bold t,$$ we can show that  the above convolution sequence converges  to  $\bold w$ (or $\nabla p$) by  appropriately choosing parameters $\eta,\beta,\gamma$ such that ${}_{\eta,\beta,\gamma}\bm{\Psi}_{\ell,k}$ is divergence-free (or curl-free).
  Thus the scaled matrix kernel  ${}_{\eta,\beta,\gamma}\bm{\Psi}_{\ell,k}^h$ provides an alternative for   vector decomposition. 
 
 However, in practical applications,  we only have   discrete vector values, i.e., $$\{\mathbf{f}(\bold j h)=(f_1(\bold j h), f_2(\bold j h),\dots, f_d(\bold j h))^T\}_{\bold j\in \mathbb{Z}^d}$$  sampled from a vector field $\bold f=(f_1,f_2,\dots,f_d)^T$, $f_s:\R^d\rightarrow \R$, at the equidistant centers $\{\bold jh\}_{\bold j\in \mathbb{Z}^d}$ over the whole space $\R^d$. In such cases, we have to  discretize the above convolution sequence ${}_{\eta,\beta,\gamma}\bm{\Psi}_{\ell,k}^h*\bold f(\bold x)$ via certain quadrature rule. In this paper,  following the   Schoenberg model \cite{Schoenberg}, we take   the rectangular rule as an example to get an ansatz:
  $$Q_{h}\mathbf{f}(\bold x)=\sum_{\bold j\in \mathbb{Z}^d} {}_{\eta,\beta,\gamma}\bm{\Psi}_{\ell,k}(\bold x/h-\bold j)\mathbf{f}(\bold jh), \ \bold x\in \R^d.$$
   This scheme includes divergence-free quasi-interpolation \cite{Gaoetal2},  curl-free quasi-interpolation, and even harmonic quasi-interpolation    by choosing different sets of $\eta, \beta, \gamma$.  Moreover, we   show that the above family of quasi-interpolants can be used for numerically computing the Helmholtz--Hodge decomposition of a vector field. 
To be specific, if we apply the above divergence-free/curl-free quasi-interpolation $Q_{h}\mathbf{f}$ directly to a  vector field $\bold f$ having the  Helmholtz--Hodge decomposition $\bold f=\bold w+\nabla p$, then it automatically approximates the corresponding divergence-free share  ($\bold w$) and curl-free share ($\nabla p$), respectively. 
 
 Finally, we extend our idea to  a vector field defined over a bounded domain $\Omega\subset \R^d$ by coupling interpolation near the boundary together with quasi-interpolation in the interior of $\Omega$, a technique proposed in \cite{Gaoetal2} for circumventing the boundary problem.
 
 The main contributions of the paper are two-fold.  First, we construct a family of structure preserving  (divergence-free/curl-free/harmonic)  quasi-interpolants for vector-valued function approximation. Second, we provide a computationally efficient and numerically stable technique for   the Helmholtz--Hodge decomposition.

The  paper is organized as follows.
Preliminaries are  provided in Section $2$ to make the paper self contained. Section $3$ focuses on constructing a quasi-interpolation based method for  the Helmholtz--Hodge decomposition of a vector field defined both over the whole space $\R^d$ and a bounded domain $\Omega$. We first constructs a matrix kernel that includes divergence-free, curl-free, and harmonic  matrix kernels as special examples.  Corresponding properties of the matrix kernel in the spatial domain and the Fourier domain are also derived. Then we apply the matrix kernel to vector decomposition. Finally, based on this matrix kernel, we construct quasi-interpolation and  derives its error estimates for the vector-valued function, its derivatives, and the vector decomposition. Section $4$  gives numerical simulations to validate theoretical convergence analysis and structure preserving properties  of our quasi-interpolation. Conclusions and discussions are presented in Section $5$.


\section{Preliminaries}\label{sec2}
We begin this section with providing some notations.
 \subsection{Notations}
Let $D$ be   a differential operator defined bv
$$D^{\alpha}=\frac{\partial^{|\alpha|}}{\partial \bold x^{\alpha}}=\frac{\partial^{|\alpha|}}{\partial x_1^{\alpha_1}\partial x_2^{\alpha_2}\dots \partial x_d^{\alpha_d}}:=\prod_{s=1}^dD_{x_s}^{\alpha_s},$$ where $D_{x_s}=\frac{\partial}{\partial x_s}$ and $\alpha=(\alpha_1,\alpha_2,\dots,\alpha_d)$ is a multi-index whose length $|\alpha|=\sum_{s=1}^d\alpha_s$.
Then the  Laplace operator $\Delta$ in $d$ variables goes as $$\Delta f(\bold x)=\sum_{j=1}^d\frac{\partial^2f(\bold x)}{\partial x_j^2}, \ \bold x=(x_1,x_2,\cdots,x_d)\in \R^d.$$
 Let $\ell$ be a positive integer and $\phi_{\ell}$ be a fundamental solution of the $\ell$-th iterated Laplace operator $\Delta^{\ell}$, namely, the equation $\Delta^{\ell}\phi_{\ell}=\delta$ holds true for any Schwartz class distribution \cite{Stein}.  We note that $\phi_{\ell}$ is unique up to polyharmonic polynomials $P_{\ell}$ in the null space of $\Delta^{\ell}$ (that is, $\Delta^{\ell}P_{\ell}=0$).  In particular,  a commonly used $\phi_{\ell}$ (which is also known as the Green's function of $\Delta^{\ell}$)  takes the form \cite{Rabut}
\begin{equation}\label{fundamentalsolution}
\phi_{\ell}(\bold x)=\|\bold x\|^{2\ell-d}[C_{\ell,d}\ln\|\bold x\|+D_{\ell,d}],
\end{equation}
where $C_{\ell,d}=0$, $D_{\ell,d}=E_{\ell,d}$ if $d$ is odd, and $C_{\ell,d}=E_{\ell,d}$, $D_{\ell,d}=0$ if $d$ is even. Here
$$E_{\ell,d}=\frac{\Gamma(d/2)}{2^{\ell}\pi^{d/2}(\ell-1)!\prod_{j=0,j\neq \ell-d/2}^{\ell-1}(2\ell-2j-d)}.$$ 
Denote by $\bold e_s$  the unit vector in the $s^{\text{th}}$ coordinate direction of $\R^d$ and let $\widetilde{\Delta}$ be a multivariate central difference  operator  whose $s^{\text{th}}$-direction component is 
$$(\widetilde{\Delta})_sf=f(\cdot-\bold e_s)-2f+f(\cdot+\bold e_s).$$
 It is straightforward to construct a polyharmonic spline $\widetilde{\Delta}^{\ell}\phi_{\ell}$ just by taking $\widetilde{\Delta}^{\ell}$ directly to the thin-plate spline $\phi_{\ell}$. Furthermore, with  $\widetilde{\Delta}^{\ell}\phi_{\ell}$ being a kernel, we can construct a quasi-interpolant in the Schoenberg form \cite{Schoenberg} $$\sum_{{\bold j}\in \mathbb{Z}^d}f({\bold j}h)\widetilde{\Delta}^{\ell}\phi_{\ell}({\bold x}/h-{\bold j}).$$   However, since the kernel $\widetilde{\Delta}^{\ell}\phi_{\ell}$ only satisfies the (incomplete) Strang--Fix condition of order two \cite{Schabackandwu}, \cite{Strang}, the above quasi-interpolant  provides a low approximation order (two) based on the celebrated results given by Jia and Lei \cite{JiaandLei}.  To overcome such a pitfall,  Rabut \cite{Rabut0, Rabut, Rabut1} provided a thorough discussion on how to construct  polyharmonic splines  with higher-order Strang--Fix conditions such that the resulting quasi-interpolant  in the Schoenberg form provides higher approximation order (that is even up to the highest order $2\ell$).
\subsection{\textit{Polyharmonic splines satisfying higher-order Strang--Fix conditions}}
 Let $k$  be a positive integer and $a_i=(-1)^i2(i!)^2/(2i+2)!$, for $i=0,1,\cdots, k-1$. Define $p_{d,\ell,k}$ to be a polynomial of degree $k\ell$ on $\R^d$ in the form
$$p_{d,\ell,k}(\bold x)=\Bigg(\sum_{i=0}^{k-1}a_i\Bigg(\sum_{s=1}^dx_s^{i+1}\Bigg)\Bigg)^{\ell}, \ \bold x=(x_1,x_2,\cdots,x_d)\in \R^d.$$  Denote by $q_{d,\ell,k}$  a truncation of $p_{d,\ell,k}$ to order  $\ell+k$.  According to Rabut \cite{Rabut0, Rabut, Rabut1}, a level-$k\ell$ polyharmonic spline satisfying the  Strang--Fix condition of order  $2k$ takes the form
\begin{equation}\label{polyharmonicspline}
\psi_{\ell,k}(\bold x)=(-1)^{\ell}q_{d,\ell,k}(\widetilde{\Delta})\phi_{\ell}(\bold x), \quad  \bold x\in \R^d.
\end{equation} 
 Obviously, $\psi_{\ell,k}$  includes  $\widetilde{\Delta}^{\ell}\phi_{\ell}$ as a special case of $k=1$.  

We go further with studying some properties of $\psi_{\ell,k}$ both in the spatial and Fourier domains.  

Define the (generalized) Fourier transform  $\hat{\phi}$ of a function $\phi$ via \cite{Stein}
$$\mathfrak{F}(\phi)(\bm{\omega}):=\hat{\phi}(\bm{\omega})=(2\pi)^{-d/2}\int_{\mathbb{R}^d}\phi(\bold x)e^{-i\bm{\omega}\cdot \bold x} d\bold x,$$
where $\bm{\omega}\cdot \bold x=\sum_{s=1}^d \omega_s x_s$. 
Then, based on the properties of the (generalized) Fourier transform, we can derive
$\hat{\phi}_{\ell}(\bm{\omega})=\|\bm{\omega}\|^{-2\ell}$, which in turn leads to
\begin{equation*}
      \widehat{\psi_{\ell,k}}(\bm{\omega})=\frac{q_{d,\ell,k}(-4\sin^2(\bm{\omega}/2))}{\|\bm{\omega}\|^{2\ell}}.
      \end{equation*}
Consequently, we have the following lemma \cite{Rabut0}.
\begin{lemma}\label{polyharmonicsplinepropert}
Let $\psi_{\ell,k}$ be a level-$k\ell$ polyharmonic spline defined in Equation  \eqref{polyharmonicspline}.  Then the decay condition
\begin{equation}\label{decayofderivatives}D^{\alpha}\psi_{\ell,k}(\bold x)=\mathcal{O}(\|\bold x\|^{-d-2k+|\alpha|})
\end{equation}
holds true for any $0\leq |\alpha|< 2k$  as  $\|\bold x\|$ tends to infinity. Moreover, the (generalized) Fourier transform $ \widehat{\psi_{\ell,k}}$ of $\psi_{\ell,k}$ satisfies  $\widehat{\psi_{\ell,k}}\in \mathcal{M}^{2k}(\R^d)$ together with the Strang--Fix conditions of order $2k$ \cite{JiaandLei}, \cite{Strang}: \begin{equation}\label{Strangfix}
\begin{cases}
\widehat{\psi_{\ell,k}}\in C^{2k}(\R^d),\\
\widehat{\psi_{\ell,k}}^{(\alpha)}(\bold 0)=\delta_{0,\alpha},\ 0\leq |\alpha|< 2k,\\
\widehat{\psi_{\ell,k}}^{(\alpha)}(2\pi \bold j)=0, \  0\leq |\alpha|< 2k, \bold j\in \mathbb{Z}^d.
\end{cases}
\end{equation}
Here the function space $ \mathcal{M}^{2k}(\R^d)$ is defined as \begin{equation}\label{eq:KernelSumProp}
    \mathcal{M}^{2k}(\R^d):=\Bigg\{\hat{\psi}\in C^{2k}(\mathbb{R}^d):\max_{|\xi|\leq 1,|\eta|\leq 1}\sum_{\bold j\neq \bold 0}|2\pi \bold j+\xi|^{|\beta|}|\hat{\psi}^{(\alpha)}(2\pi \bold j +\eta)|<\infty, ~0\leq |\alpha|\leq 2k, ~0\leq|\beta|\leq 2k-1\Bigg\}.
\end{equation}
\end{lemma}
More importantly, according to \cite[Theorem 2.1]{LeijiaandCheney},  we can derive the following simultaneous error estimates.
 \begin{lemma}\label{convolution}
 Let $\psi_{\ell, k}^h(\bold x)=h^{-d}\psi_{\ell, k}(\bold x/h)$ be a scaled version of $\psi_{\ell, k}$. Define the   convolution of $f$ with $\psi_{\ell,k}^h$ in the form $$f*\psi_{\ell,k}^h(\bold x)=\int_{\R^d}f(\bold t)\psi_{\ell, k}^h(\bold x-\bold t)d\bold t.$$  Then we have   error estimates
$$\|(D^{\alpha}f)*\psi_{\ell,k}^h-D^{\alpha}f\|_p = \mathcal{O}(h^{2k})$$
 for any $0\leq |\alpha|\leq \min\{2k, r-2k\}$ and any function $f$ in the Sobolev space 
 \begin{equation}\label{fconditions}
 \mathcal{S}_p^r(\R^d):=\Bigg\{f\in W_p^r(\R^d):\bm{\omega}^{\alpha}\hat{f}(\bm{\omega})\in L^p(\R^d), |\alpha|=r, \ r>d/2\Bigg\}. 
\end{equation} 
Moreover,  the quasi-interpolant in the Schoenberg form $$Q_{h}f(\bold x)=\sum_{{\bold j}\in \mathbb{Z}^d}f({\bold j}h)\psi_{\ell,k}({\bold x}/h-{\bold j})$$ provides simultaneous approximation orders  (to both the approximated function and its derivatives) as  \begin{equation}\label{simultaneousofpol}
\|D^{\alpha}Q_hf-D^{\alpha}f\|_{p,\R^d}=\mathcal{O}(h^{2k-|\alpha|}),\  0\leq |\alpha|<\min\{2k,r-2k,2\ell -d-1\}.
\end{equation} 
 \end{lemma}
 \subsection{The Helmholtz--Hodge decomposition}
The Helmholtz--Hodge decomposition of a vector field $\bold f$ defined in the whole space $\R^d$ is straightforward to describe via the Fourier transform. More precisely, let $1 \le p \le \infty$  and denote by $L^p(\R^d)$  the Banach space consisting of all Lebesgue measurable functions $f$ on $\R^d$ for which $\|f\|_{L^p(\R^d)} < \infty$. Let $\bold L^p(\R^d)$ be the space of all vector fields with components in  $L^p(\R^d)$. Then the vector field  $\bold f \in \bold L^2(\R^d)$ is divergence-free if and only if $\bm{\omega}^T \hat {\bold f}(\bm{\omega}) = 0$ almost everywhere, and  is curl-free only if $ \hat {\bold f}(\bm{\omega}) = \bm{\omega} \hat {p}$ for some $p$ in the classical Sobolev space $ W^1_2(\R^d)$. 
Hence, we can define the following
projection operators via the inverse Fourier transform $\mathfrak{F}^{-1}:\bold L^2(\R^d)\to\bold L^2(\R^d)$ as

\begin{equation}\label{projectionoperators_Fouriertransform}
    \mathcal{P}_{\text{div}} \bold f := \mathfrak{F}^{-1}\left(\left(\bold I_d-\frac{\bm{\omega}\bm{\omega}^T}{\|\bm{\omega}\|^2}\right) \hat{\bold f}(\bm{\omega})\right), \quad
        \mathcal{P}_{\text{curl}} \bold f := \mathfrak{F}^{-1}\left(\left( \frac{\bm{\omega}\bm{\omega}^T}{\|\bm{\omega}\|^2}\right) \hat{\bold f}(\bm{\omega})\right).
\end{equation}
We can show that the projection $ \mathcal{P}_{\text{div}} \bold f$ is divergence-free,
$\mathcal{P}_{\text{curl}} \bold f$ is curl-free, and that $ \mathcal{P}_{\text{div}} \bold f \perp  \mathcal{P}_{\text{curl}} \bold f$. Furthermore, we can verify that $\bold f = \mathcal{P}_{\text{div}} \bold f  +\mathcal{P}_{\text{curl}} \bold f $ is a unique decomposition of $\bold f $  into divergence-free and curl-free vector fields which are orthogonal in the $\bold L^2(\R^d)$ sense.

On the other hand,  the Helmholtz--Hodge decomposition of $\bold f$ defined over a bounded compact domain is much more involved as described in the following lemma (that is Proposition 2.1 in \cite{fw}).

\begin{lemma}\label{lem:LerayProjection}
 Let $\Omega \subset \R^d$ be a connected Lipschitz domain, $\bold f \in \bold L^2(\Omega)$ be such that
 $\nabla \cdot \bold f \in L^2(\Omega)$ and let $g\in H^{1/2}(\partial \Omega)$ satisfy $\langle g,1 \rangle_{L^2(\Omega)} = 0$. Then one has the unique decomposition $\bold f = \bold w + \nabla p$, where
 $p\in W_2^1(\Omega)$, and $\bold w \in \bold L^2(\Omega)$ with $\nabla \cdot \bold w = 0$ and $\bold w \cdot \bold n = g$ on $\partial \Omega$. Moreover, the function $p$ is uniquely determined up to a constant, and satisfies the bound
 \begin{equation}
     |p|_{W_2^1(\Omega)} = \|\nabla p \|_{L^2(\Omega)}\leq C(\|\nabla \cdot \bold f\|_{L^2(\Omega)} + \|\bold f \cdot \bold n - g\|_{H^{-1/2}(\partial \Omega)}),
 \end{equation}
 where $C$ is some constant independent of $\bold f$. 
\end{lemma}
\begin{remark}\label{rem:boundaryconditions}
In particular, when $g = 0$, the vector decomposition defined in Lemma \ref{lem:LerayProjection} gives rise to the Leray projector
\[\mathcal P_L \bold f := \bold w,\]
whoes orthogonal complement  is 
\[\mathcal P^{\perp}_L \bold f := \nabla p.\]
Consequently, $\bold f = \mathcal{P}_{L} \bold f  +\mathcal{P}^{\perp}_{L} \bold f $ is a unique decomposition of $\bold f $  into divergence-free and curl-free vector fields which are orthogonal in the $ \bold L^2(\Omega)$ sense.
\end{remark}
With the help of the above two projection techniques, we shall construct  a quasi-interpolation based approach for  Helmholtz--Hodge decomposition of a vector field $\bold f$ defined both in the whole space and bounded domains, respectively.


\section{The Main Results}\label{sec3}

This section first constructs  a family of (divergence-free, curl-free, and harmonic)    quasi-interpolants for vector-valued function approximation and derives simultaneous approximation orders to both the approximated function and its derivatives. Then we discuss the application of  divergence-free and curl-free  quasi-interpolants to the Helmholtz--Hodge decomposition. We start by constructing a general matrix kernel with certain properties.
\subsection{{ Divergence-free, curl-free, and harmonic matrix kernels}}
 Motivated by the papers \cite{DoduandRabut}, \cite{Fuselier, Fuselier1}, \cite{Wendland1}, it is natural to choose divergence-free  and curl-free matrix kernels in the form $[\Delta \bold I_d-\nabla\nabla^T]\Phi$  and $ \nabla \nabla^T\Phi$, respectively, with some \textit{a priori} determined function $\Phi$. In what follows we illustrate how to construct a general  matrix kernel  including divergence-free, curl-free, and even harmonic matrix kernels. In addition, to make  the following exposition transparent, we only demonstrate our construction process from commonly used polyharmonic splines as an example, but it should be noted that the basic idea is rather general and is still valid for many other functions (e.g., the fundamental solutions of some distributional operators  considered in \cite{DynandJackRon}). In the sequel, we assume that the components of $\bold f$ are all in the Sobolev space $W_p^r(\R^d)$, $r>d/2$.

   The divergence-free matrix
   kernel introduced in \cite{Gaoetal2} is given as
\begin{equation}\label{div}
 [\Delta {\bf I}_d - \nabla \nabla^T]\Delta^\ell\phi_{\ell+1}.
\end{equation}
Here $\phi_{\ell+1}$ is a fundamental solution of $\Delta^{\ell+1}$ in the distributional sense.  Moreover, we have  
\begin{equation}\label{repro1}
[\Delta {\bf I}_d - \nabla \nabla^T]\Delta^\ell \phi_{\ell+1}*{\bf f} = {\bf f}
\end{equation}
for any divergence-free vector-valued function  ${\bf f}$  (${\nabla \cdot {\bf f} = 0}$). Besides, motivated by the identity $\nabla \nabla^T = \Delta{\bf I}_d - \nabla \times \nabla \times$ (here and throughout, $\nabla \times \nabla \times$ is interpreted as a matrix-valued differential operator), we can  construct a curl-free matrix kernel 
\begin{equation}  \label{curl}
\nabla \nabla^T\Delta^\ell \phi_{\ell+1},
\end{equation}
 such that 
 for any curl-free vector-valued function ${\bf f}$ ( ${\nabla \times \bf f = 0}$), we have
\begin{equation}\label{repro2}
\begin{split}
 \nabla \nabla^T\Delta^\ell \phi_{\ell+1}*{\bf f}
 &= [\Delta {\bf I}_d - \nabla \times \nabla \times]\Delta^\ell\phi_{\ell+1}*{\bf f}\\
 &=\Delta^{\ell +1 }{\bf I}_d \phi_{\ell+1}*{\bf f} -   \nabla \times \nabla \times \Delta^\ell\phi_{\ell+1}*{\bf f}\\
&= {\bf f} -  \int_{\mathbb{R}^d} \nabla  \Delta^\ell\phi_{\ell+1}(\bold x-\bold t)\times (\nabla \times {\bf f}(\bold t)) d \bold t\\
&={\bf f}.
    \end{split}
    \end{equation}

   Furthermore, coupling the above two matrix kernels, we can even construct a harmonic matrix kernel (both divergence-free and curl-free)  in the form 
\begin{equation}\label{harmon}
 [\Delta {\bf I}_d - \nabla \nabla^T][\Delta{\bf I}_d - \nabla \times \nabla \times ]\Delta^\ell\phi_{\ell+2}.
\end{equation}
Similarly, we can get the reproduction identity
\begin{equation}\label{repro3}
\begin{split}
 &[\Delta {\bf I}_d - \nabla \nabla^T][\Delta{\bf I}_d - \nabla \times \nabla \times ]\Delta^\ell \phi_{\ell+2}*{\bf f}\\
 &=\Delta^{\ell+2}{\bf I}_d\phi_{\ell+2}*{\bf f} -\Delta^{\ell+1}{\bf I}_d\nabla\times\nabla\times(\phi_{\ell+2}*{\bf f}) \\
&  -\Delta^{\ell+1}{\bf I}_d\nabla\nabla^T(\phi_{\ell+2}*{\bf f}) +\Delta^{\ell}\nabla\nabla^T\nabla\times\nabla\times(\phi_{\ell+2}*{\bf f})\\
&={\bf f}-\Delta^{\ell+1}{\bf I}_d\int_{\mathbb{R}^d}\nabla\phi_{\ell+2}(\bold x-\bold t)\times\nabla\times{\bf f}(\bold t)d\bold t\\
&  -\Delta^{\ell+1}{\bf I}_d\int_{\mathbb{R}^d}\nabla\phi_{\ell+2}(\bold x-\bold t)\nabla\cdot{\bf f}(\bold t)d\bold t\\
&+\Delta^{\ell}\int_{\mathbb{R}^d}\nabla\nabla^T\nabla \phi_{\ell+2}(\bold x-\bold t)\times\nabla\times{\bf f}(\bold t)d\bold t \\
 &= {\bf f}
 \end{split}
\end{equation}
  for any harmonic vector field $\bold f$.

All together,  the above three  kernels (defined in Equations  \eqref{div},  \eqref{curl}, \eqref{harmon}) can be expressed via a general form (up to a constant) as
\begin{equation}\label{thefinal}
 [\eta\Delta {\bf I}_d - \beta\nabla \nabla^T](\nabla \nabla^T)^{\gamma} \Delta^\ell\phi_{\ell+\gamma+\max\{\eta,\beta\}}, \  \eta, \beta, \gamma \in\{0,1\}, \eta + \beta+ \gamma\neq 0.
  \end{equation}
We note that the reproduction properties (Equation \eqref{repro1}, Equation \eqref{repro2},  Equation \eqref{repro3}) demonstrate that the matrix kernel (\ref{thefinal})  behaves like a Dirac distribution when it is convolved with a vector field $\bold f$ satisfying the corresponding (divergence-free, curl-free, harmonic) properties. However,  the above matrix kernel is  \textbf{not} user friendly. For practical applications, we   discretize the differential operator $\Delta^{\ell}$   in Equation \eqref{thefinal} using the difference operator $(-1)^{\ell}q_{d,\ell,k}(\widetilde{\Delta})$ introduced in Subsection $2.2$ to get 
  \begin{equation}\label{divergencefreekernl}
  {}_{\eta,\beta,\gamma}\bm{\Psi}_{\ell,k}:=(-1)^{\ell}q_{d,\ell,k}(\widetilde{\Delta})[\eta\Delta {\bf I}_d - \beta\nabla \nabla^T](\nabla \nabla^T)^{\gamma}\phi_{\ell+\gamma+\max\{\eta,\beta\}}, \  \eta, \beta, \gamma \in\{0,1\}, \eta+ \beta+ \gamma\neq 0.
  \end{equation}
One can verify that  ${}_{\eta,\beta,\gamma}\bm{\Psi}_{\ell,k}$ includes \textbf{divergence-free} kernel (when $(\eta,\beta,\gamma)=(1,1,0)$), \textbf{curl-free} kernels (when $(\eta,\beta,\gamma)\in\{(0,1,1), (1,0,1), (0,1,0), (0,0,1)\}$ up to a constant $-1$),  and \textbf{harmonic} kernel (when $(\eta,\beta,\gamma)=(1,1,1)$). 

We proceed with  studying some other properties of the  matrix kernel \eqref{divergencefreekernl}. Similar to the above discussions, we can derive the reproduction identity
\begin{equation}\label{newrepro}
{}_{\eta,\beta,\gamma}\bm{\Psi}_{\ell,k}*\bold f=\bold I_d\psi_{\ell,k}*\bold f
\end{equation}
  for any vector-valued function $\bold f$ satisfying corresponding divergence-free/curl-free/harmonic properties. In addition, by noting  that derivatives of a divergence-free/curl-free/harmonic function are all divergence-free/curl-free/harmonic,  we can even get a general reproduction identity 
\begin{equation}\label{relationofsf}
 {}_{\eta,\beta,\gamma}\bm{\Psi}_{\ell,k}*(D^{\alpha}\mathbf{f})=\bold I_d\psi_{\ell,k}*(D^{\alpha}\mathbf{f}).
\end{equation}
More importantly, the following simultaneous approximation orders hold true.
\begin{lemma}\label{lem:convolutionerror}
 Define the $L_p$-norm  of a vector-valued function $\bold f=(f_1,f_2,\cdots, f_d)$ via $\|\mathbf{f}\|_{p,\R^d}=(\sum_{s=1}^d\|f_s\|_{p,\R^d}^p)^{1/p}$.   Denote by ${}_{\eta,\beta,\gamma}\bm{\Psi}_{\ell,k}^h(\bold x):=h^{-d}{}_{\eta,\beta,\gamma}\bm{\Psi}_{\ell,k}({\bold x}/h)$  a scaled version of  the matrix kernel ${}_{\eta,\beta,\gamma}\bm{\Psi}_{\ell,k}$ defined in Equation \eqref{divergencefreekernl}. Then, for any  divergence-free/curl-free/harmonic vector field $\mathbf{f}$ whose components   $f_s\in \mathcal{S}_p^r(\R^d)$, $s=1,2,\dots,d$,  we can  construct a  corresponding divergence-free/curl-free/harmonic matrix kernel ${}_{\eta,\beta,\gamma}\bm{\Psi}_{\ell,k}$ by appropriately choosing   values of parameters $\eta,\beta,\gamma$, such that
the  simultaneous error estimates
 \begin{equation}\label{conerror}
        \| {}_{\eta,\beta,\gamma}\bm{\Psi}_{\ell,k}^h*(D^{\alpha}\mathbf{f})-D^{\alpha}\mathbf{f}\|_{p,\R^d}=\mathcal{O}(h^{2k})
    \end{equation}
 hold true for any $0\leq |\alpha|\leq \min\{2k, r-2k\}$.
\end{lemma}
\begin{proof}
A special case of the estimate (\ref{conerror}) was proven in \cite{Gaoetal2} for a   divergence-free vector field and corresponding divergence-free matrix kernel. Here we consider the general case.

Based on the  general reproduction identity (that is Equation \eqref{newrepro})
\begin{equation*}
 {}_{\eta,\beta,\gamma}\bm{\Psi}_{\ell,k}*(D^{\alpha}\mathbf{f})=\bold I_d\psi_{\ell,k}*(D^{\alpha}\mathbf{f}),
\end{equation*} we only  need to derive simultaneous error estimates
 $$\|\bold I_d\mathbf{\psi}_{\ell,k}^h*(D^{\alpha}\mathbf{f})-D^{\alpha}\mathbf{f}\|_{p,\R^d},$$
where $\mathbf{\psi}_{\ell,k}^h(\bold x)=h^{-d}\mathbf{\psi}_{\ell,k}(\bold x/h)$ is a scaled version of the level-$k\ell$ polyharmonic spline $\mathbf{\psi}_{\ell,k}$ defined in Equation \eqref{polyharmonicspline}. Moreover, the definition of the  $L_p$-norm of a vector-valued function implies that  it suffices to show that $\|\bold I_d\mathbf{\psi}_{\ell,k}^h*(D^{\alpha}\mathbf{f})-D^{\alpha}\mathbf{f}\|_{p,\R^d}=\mathcal{O}(h^{2k})$ holds componentwise. Namely,
 \begin{equation*}
\|(D^{\alpha}f_s)*\psi_{\ell,k}^h-D^{\alpha}f_s\|_{p,\R^d}=\mathcal{O}(h^{2k}), \ \text{for}\ s=1,2,\dots, d.
\end{equation*}
This follows directly from Lemma \ref{convolution}.   Thus the lemma holds.
\end{proof}

Next we study properties of the matrix kernel $ {}_{\eta,\beta,\gamma}\bm{\Psi}_{\ell,k}$ in the Fourier domain.
According to the theory of the Fourier transform \cite{Stein}, we have
\begin{equation}\label{fouriertrans}
{}_{\eta,\beta,\gamma}\widehat{\bm{\Psi}_{\ell,k}}(\bm{\omega})=\widehat{\psi_{\ell,k}}(\bm{\omega})[\eta\|\bm{\omega}\|^2\bold I_d-\beta\bm{\omega}\bm{\omega}^T](\bm{\omega}\bm{\omega}^T)^{\gamma}/\|\bm{\omega}\|^{2+\gamma+\max\{\eta,\beta\}}.
\end{equation}
 Besides, since $\psi_{\ell,k}$ satisfies Strang--Fix conditions of order $2k$ (that is Equation \eqref{Strangfix}), we can get
\begin{equation}\label{fouriertrans2}
D^{\alpha}{}_{\eta,\beta,\gamma}\widehat{\bm{\Psi}_{\ell,k}}(2\pi \bold j)=0, \quad \bold j\in \mathbb{Z}^d/\{\bold 0\}, \ \text{for} \ 0\leq |\alpha|<2k.
\end{equation}
This equation will play a vital role in establishing the convergence analysis in Subsection \ref{constructionRd}.

We go further with using  the  matrix kernel ${}_{\eta,\beta,\gamma}\bm{\Psi}_{\ell,k}$ for kernel projections.

\subsection{Kernel projections}
Note that Lemma \ref{lem:convolutionerror} is only valid for a  vector field $\bold f$ that has the same  (divergence-free, curl-free, harmonic) property of the matrix kernel ${}_{\eta,\beta,\gamma}\bm{\Psi}_{\ell,k}$.  Fortunately, the Helmholtz--Hodge decomposition  states that any sufficiently smooth vector field, ${\bf f}$,
can be decomposed as the sum of a divergence-free component, ${\bf w}$, and a curl-free component, ${\nabla p}$, that is,
\[ {\bf f}  = {\bf w}  + {\nabla p}.\]
Under such a decomposition, one can show that the convolution of ${\bf f }$ with the divergence-free kernel
returns only the divergence-free part of ${\bf f }$:
\begin{equation*}
\begin{array}{rcl}
 [\Delta {\bf I}_d - \nabla \nabla^T]\Delta^\ell \phi_{\ell+1}*{\bf f}& = &  [\Delta {\bf I}_d - \nabla \nabla^T]\Delta^\ell\phi_{\ell+1}*({\bf w}  + {\nabla p})\\
                                                                                                     & = & {\bf w} + \Delta^\ell [\Delta {\bf I}_d - \nabla \nabla^T]\phi_{\ell+1}*{\nabla p}\\
                                                                                                     & = & {\bf w} + \Delta^{\ell +1} {\bf I}_d\phi_{\ell+1}*{\nabla p} - \Delta^{\ell }\nabla \nabla^T\phi_{\ell+1}*{\nabla p}\\
                                                                                                     & = & {\bf w} + ({\nabla p} - {\nabla p})\\
                                                                                                     & = & {\bf w}.
 \end{array}
\end{equation*}
Therefore, comparing the above calculation to Equation  \eqref{relationofsf} and setting  $ \bm{\Psi}_{\ell,k}^{div}={}_{1,1,0}\bm{\Psi}_{\ell,k}$, we have the \textbf{divergence-free projection}:
\begin{equation}\label{relationofsf:projection}
 \bm{\Psi}_{\ell,k}^{div}*(D^{\alpha}\mathbf{f})=\bold I_d\psi_{\ell,k}*(D^{\alpha}\mathbf{w}).
\end{equation}
Likewise, the convolution of ${\bf f }$ with the curl-free kernel
returns only the curl-free part of ${\bf f }$:

\begin{equation*}
\begin{array}{rcl}
 \Delta^\ell [\Delta {\bf I}_d - \nabla \times \nabla \times]\phi_{\ell+1}*{\bf f}& = & \Delta^\ell [\Delta {\bf I}_d -\nabla \times \nabla \times]\phi_{\ell+1}*({\bf w}  + {\nabla p})\\
                                                                                                     & = & {\nabla p} + \Delta^\ell [\Delta {\bf I}_d - \nabla \times \nabla \times\phi_{\ell+1}*{\bf w}\\
                                                                                                     & = & {\nabla p} + \Delta^{\ell +1} {\bf I}_d\phi_{\ell+1}*{\bf w} - \Delta^{\ell }\nabla \times \nabla \times\phi_{\ell+1}*{\bf w}\\
                                                                                                     & = & {\nabla p} + ({\bf w} - {\bf w})\\
                                                                                                     & = & {\nabla p}.
 \end{array}
\end{equation*}
Similarly, one can get the \textbf{curl-free projection}:
\begin{equation}\label{relationofsf2:projection}
\bm{\Psi}^{curl}_{\ell,k}*(D^{\alpha}\mathbf{f}) =\bold I_d\psi_{\ell,k}*(D^{\alpha}(\nabla p)).
\end{equation}
Here
\begin{equation}\label{curlfreeprojectionkernel}
\bm{\Psi}^{curl}_{\ell,k}(\bold x) \in \left\{ -{}_{0,1,1}\bm{\Psi}_{\ell,k}(\bold x), \,
{}_{1,0,1}\bm{\Psi}_{\ell,k}(\bold x),\, -{}_{0,1,0}\bm{\Psi}_{\ell,k}(\bold x),\, {}_{0,0,1}\bm{\Psi}_{\ell,k}(\bold x)\right\}.
\end{equation}
Consequently, parallel with Lemma \ref{lem:convolutionerror},  we have the following simultaneous error estimates for a general vector field $\bold f$. 
\begin{lemma}\label{projectionconvolutionerror}
 Let $\mathbf{f}$ be a vector field whose components satisfy $f_s\in \mathcal{S}_p^r(\R^d)$, $s=1,2,\dots,d$.   Suppose that  $\bold f$ has the Helmholtz--Hodge decomposition: $\mathbf{f} = {\bf w}  + {\nabla p}$, where $\nabla^T {\bf w} = 0$ and $\nabla\times ({\nabla p}) = \bold 0$ with $\partial p / \partial x^s\in \mathcal{S}_p^r(\R^d)$, $s=1,2,\dots,d$.
 Define its $L_p$-norm via $\|\mathbf{f}\|_{p,\R^d}=(\sum_{s=1}^d\|f_s\|_{p,\R^d}^p)^{1/p}$. Let $\bm{\Psi}_{\ell,k}^{div,h}(\bold x):=h^{-d}\bm{\Psi}_{\ell,k}^{div}({\bold x}/h)$ and $\bm{\Psi}_{\ell,k}^{curl,h}(\bold x):=h^{-d}\bm{\Psi}_{\ell,k}^{curl}({\bold x}/h)$ be  scaled  matrix kernels.
 Then the following simultaneous error estimates hold true for any $0\leq |\alpha|\leq \min\{2k, r-2k\}$:
\begin{enumerate}
    \item For the divergence-free projection
     \begin{equation}\label{projection}
        \|\bm{\Psi}_{\ell,k}^{div,h}*(D^{\alpha}\mathbf{f})-D^{\alpha}\mathbf{w}\|_{p,\R^d}=\mathcal{O}(h^{2k}).
    \end{equation}
    \item For the curl-free projection
         \begin{equation}\label{conerror2}
        \|\bm{\Psi}_{\ell,k}^{curl,h}*(D^{\alpha}\mathbf{f})-D^{\alpha}(\nabla{ p)}\|_{p,\R^d}=\mathcal{O}(h^{2k}).
    \end{equation}       
\end{enumerate}
 \end{lemma}

\begin{proof}
  Observing that the Helmholtz--Hodge decomposition holds true for any  $\bold f$ whose components are all in $ \mathcal{S}_p^r(\R^d) \subset L^2(\R)$,   the proof is similar to that of Lemma \ref{lem:convolutionerror}.
\end{proof}
\subsection{Construction of  quasi-interpolation schemes}\label{constructionRd}
Subsection $3.1$  demonstrates that convolution of a vector-valued (divergence-free, curl-free, harmonic) function $\bold f$ (as well as its derivatives) with corresponding scaled (divergence-free, curl-free, harmonic)  matrix kernel ${}_{\eta,\beta,\gamma}\bm{\Psi}_{\ell,k}^h$ provides a fair approximation to $\bold f$ (as well as its derivatives).  Subsection $3.2$, on the other hand, shows that convolution of scaled divergence-free/curl-free matrix kernel ${}_{\eta,\beta,\gamma}\bm{\Psi}_{\ell,k}^h$ with a \textbf{general} vector field $\bold f$  (as well as its derivatives) approximates the divergence-free/curl-free part  (as well as its derivatives) of the Helmholtz--Hodge decomposition of $\bold f$, respectively. However, in practical applications, we only have discrete values of a vector field. To make the above results more available,  we have to discretize these convolution sequences with certain quadrature rules.
This subsection proceeds with  discretization of the above convolution sequences to construct corresponding quasi-interpolation schemes. We first construct  quasi-interpolation for vector-valued function approximation defined in the whole space $\R^d$ and study its simultaneous approximation orders. Then we extend our idea to construct corresponding quasi-interpolation schemes defined over a bounded domain $\Omega\subset \R^d$.
\subsubsection{Quasi-interpolation defined in the whole space $\R^d$}
 Let $\{\bold jh\}_{\bold j\in \mathbb{Z}^d}$ be equidistant centers  over the whole space $\R^d$ and $\{\mathbf{f}(\bold jh)\}_{\bold j\in \mathbb{Z}^d}$ be corresponding sampling data.  We  first discretize the convolution sequence $${}_{\eta,\beta,\gamma}\bm{\Psi}_{\ell,k}^h*\bold f(\bold x)=\int_{\R^d}{}_{\eta,\beta,\gamma}\bm{\Psi}_{\ell,k}^h(\bold x-\bold t)\bold f(\bold t)d\bold t$$ via the rectangular rule to  construct a quasi-interpolation scheme in the Schoenberg form
\begin{equation}\label{quasiinterpolationofourpaper}
Q_h\mathbf{f}(\bold x):=\sum_{\bold j\in \mathbb{Z}^d}{}_{\eta,\beta,\gamma}\bm{\Psi}_{\ell,k}(\bold x/h-\bold j)\mathbf{f}(\bold jh).
\end{equation}
Furthermore, we can even use $D^{\alpha}Q_h\mathbf{f}$ directly to approximate the corresponding divergence-free/curl-free part (as well as its derivatives) of $\mathbf{f}$ and get the following simultaneous error estimates.

\begin{theorem}\label{errorestimateforderi}
Let $\mathbf{f}$ be a vector field whose components $f_s\in \mathcal{S}_p^r(\R^d)$, $s=1,2,\dots, d$. Let $\bold g$ be a corresponding vector field satisfying: $\bold g=\bold f$ if $\bold f$ is divergence-free/curl-free/ harmonic, or else, $\bold g$ denotes the divergence-free part ($\bold g=\bold w$) or curl-free part ($\bold g=\nabla p$) of $\bold f$.  Let $Q_h\bold f$ be defined as in Equation \eqref{quasiinterpolationofourpaper} with ${}_{\eta,\beta,\gamma}\bm{\Psi}_{\ell,k}$ being a  matrix kernel constructed in Equation \eqref{divergencefreekernl}. Then for any $\bold g$, we can construct a corresponding matrix kernel ${}_{\eta,\beta,\gamma}\bm{\Psi}_{\ell,k}$ (having the same property as that of $\bold g$ by appropriately choosing values of $\eta,\beta,\gamma$), such that the following simultaneous error estimates 
$$\|D^{\alpha}Q_h\mathbf{f}-D^{\alpha}\mathbf{g}\|_{p,\R^d}=\mathcal{O}(h^{2k-|\alpha|})$$ holds true for any $0\leq |\alpha|< \min\{2k, r-2k,2\ell-d-1\}$.
\end{theorem}
\begin{proof}
 Fix a point $\bold x\in \R^d$. We first rewrite  $D^{\alpha}Q_h\mathbf{f}(\bold x)-D^{\alpha}\mathbf{g}(\bold x)$ using the Fourier transform as
\begin{equation*}
\begin{split}
&D^{\alpha}Q_h\mathbf{f}(\bold x)-D^{\alpha}\mathbf{g}(\bold x)\\
&=(-i)^{|\alpha|}(2\pi)^{-d/2}\int_{\R^d}\Bigg[h^{-|\alpha|}\sum_{\bold j\in \mathbb{Z}^d}(h\bm{\omega}+2\bold j\pi)^{\alpha}\widehat{{}_{\eta,\beta,\gamma}\bm{\Psi}_{\ell,k}}(h\bm{\omega}+2\bold j\pi)\hat{\mathbf{f}}(\bm{\omega})e^{2i\pi \bold j\bold x/h}-\bm{\omega}^{\alpha}\bold I_d\hat{\mathbf{g}}(\bm{\omega})\Bigg]e^{i\bold x\bm{\omega}}d\bm{\omega}.
\end{split}
\end{equation*} Then we split the above integrand into two parts to get
\begin{equation*}
\begin{split}
&|D^{\alpha}Q_h\mathbf{f}(\bold x)-D^{\alpha}\mathbf{g}(\bold x)|\\
&\leq \Bigg|(2\pi)^{-d/2}\int_{\R^d}[\hat{\mathbf{f}}(\bm{\omega})\widehat{{}_{\eta,\beta,\gamma}\bm{\Psi}_{\ell,k}}(h\bm{\omega})-\bold I_d\hat{\mathbf{g}}(\bm{\omega})]\bm{\omega}^{\alpha}e^{i\bold x\bm{\omega}}d\bm{\omega}\Bigg|\\
&+h^{-|\alpha|}(2\pi)^{-d/2}\Bigg|\int_{\R^d}\sum_{\bold j\in \mathbb{Z}^d/\{\bold 0\}}(h\bm{\omega}+2\bold j\pi)^{\alpha}\widehat{{}_{\eta,\beta,\gamma}\bm{\Psi}_{\ell,k}}(h\bm{\omega}+2\bold j\pi) \hat{\mathbf{f}}(\bm{\omega})e^{i\bold x\bm{\omega}}e^{2i\bold j\pi \bold x/h}d\bm{\omega}\Bigg|\\
&\leq|{}_{\eta,\beta,\gamma}\bm{\Psi}_{\ell,k}^h*(D^{\alpha}\mathbf{f})(\bold x)-D^{\alpha}\mathbf{g}(\bold x)|\\
&+h^{-|\alpha|}(2\pi)^{-d/2}\int_{\R^d}\sum_{\bold j\in \mathbb{Z}^d/\{\bold 0\}}|(h\bm{\omega}+2\bold j\pi)^{\alpha}\widehat{{}_{\eta,\beta,\gamma}\bm{\Psi}_{\ell,k}}(h\bm{\omega}+2\bold j\pi) \hat{\mathbf{f}}(\bm{\omega})|d\bm{\omega}.
\end{split}
\end{equation*}
This together with Lemma \ref{lem:convolutionerror} and Lemma \ref{projectionconvolutionerror} leads to
$$\|{}_{\eta,\beta,\gamma}\bm{\Psi}_{\ell,k}^h*(D^{\alpha}\mathbf{f})-D^{\alpha}\mathbf{g}\|_{p,\R^d}=\mathcal{O}(h^{2k}).$$
Moreover, noting that $\widehat{{}_{\eta,\beta,\gamma}\bm{\Psi}_{\ell,k}}$ satisfies Equation \eqref{fouriertrans2}, we  expand $\widehat{{}_{\eta,\beta,\gamma}\bm{\Psi}_{\ell,k}}(h\bm{\omega}+2\bold j\pi)$ around $2\bold j\pi$ via Taylor expansion up to order $2k$  to get
\begin{equation*}
\begin{split}
&(2\pi)^{-d/2}h^{-|\alpha|}\int_{\R^d}\sum_{\bold j\in \mathbb{Z}^d/\{0\}}|(h\bm{\omega}+2\bold j\pi)^{\alpha}\widehat{{}_{\eta,\beta,\gamma}\bm{\Psi}_{\ell,k}}(h\bm{\omega}+2\bold j\pi) \hat{\mathbf{f}}(\bm{\omega})|d\bm{\omega}\\
&=\frac{(2\pi)^{-d/2}h^{2k-|\alpha|}}{(2k)!}\int_{\R^d}\sum_{\bold j\in \mathbb{Z}^d/\{\bold 0\}}|(h\bm{\omega}+2\bold j\pi)^{\alpha}|\cdot|\widehat{{}_{\eta,\beta,\gamma}\bm{\Psi}_{\ell,k}}^{(2k)}(\bm{\xi}_\bold j+2\bold j\pi)|\cdot |\bm{\omega}^{2k} \hat{\mathbf{f}}(\bm{\omega})|d\bm{\omega}, \end{split}
 \end{equation*}
for any $\ \bm{\xi}_\bold j\in (2\bold j\pi, 2\bold j\pi+h\bm{\omega})$.
This together with Equation \eqref{fouriertrans} and the fact that $\widehat{\psi_{\ell,k}}\in \mathcal{M}^{2k}(\R^d)$ yields
$$\sum_{\bold j\in \mathbb{Z}^d/\{\bold 0\}}|(h\bm{\omega}+2\bold j\pi)^{\alpha}|\cdot|\widehat{{}_{\eta,\beta,\gamma}\bm{\Psi}_{\ell,k}}^{(2k)}(\bm{\xi}_{\bold j}+2\bold j\pi)|<\infty, \quad  0\leq |\alpha|< \min\{2k, r-2k,2\ell-d-1\}.$$ In addition, the conditions that  $f_s\in \mathcal{S}_p^r(\R^d)$, $s=1,2,\dots, d$, yield the inequality $$\int_{\R^d}|\bm{\omega}^{2k}\mathbf{f}(\bm{\omega})|d\bm{\omega}<\infty,\quad 2k\leq \min\{2\ell, r\},$$
which in turn gives
 $$h^{-|\alpha|}(2\pi)^{-d/2}\int_{\R^d}\sum_{\bold j\in \mathbb{Z}^d/\{\bold 0\}}|(h\bm{\omega}+2\bold j\pi)^{\alpha}|\cdot|\widehat{{}_{\eta,\beta,\gamma}\bm{\Psi}_{\ell,k}}(h\bm{\omega}+2\bold j\pi) \hat{\mathbf{f}}(\bm{\omega})|d\bm{\omega}\\
=\mathcal{O}(h^{2k-|\alpha|}).$$
 Therefore, by combining the above two parts, we have
$\|D^{\alpha}Q_h\mathbf{f}-D^{\alpha}\mathbf{g}\|_{p,\R^d}=\mathcal{O}(h^{2k-|\alpha|})$.
\end{proof}

\begin{remark}\label{numericalcomput}
The above theorem demonstrates that if the vector field $\bold f$ is divergence-free/curl-free/harmonic, then we can construct a matrix kernel ${}_{\eta,\beta,\gamma}\bm{\Psi}_{\ell,k}$ with the same property of  $\bold f$ (by appropriately choosing values of $\eta,\beta,\gamma$), such that  the resulting quasi-interpolant $Q_h\bold f$ approximates both the field as well as its derivatives. Otherwise,  the quasi-interpolant $Q_h\bold f$ with a divergence-free/curl-free matrix kernel   approximates corresponding  divergence-free/curl-free part (as well as derivatives) of the Helmholtz--Hodge decomposition of  $\bold f$.  
\end{remark}

We go further with showing  how to use the quasi-interpolation scheme \eqref{quasiinterpolationofourpaper} for computing the Helmholtz--Hodge decomposition of a vector field $\bold f$ defined on $\R^d$ and establishing  corresponding approximation orders. To this end, we first define two quasi-interpolants:
\begin{equation}\label{quasiinterpolatio_divfreeprojection}
Q^{div}_h\mathbf{f}(\bold x):=\sum_{\bold j\in \mathbb{Z}^d}\bm{\Psi}^{div}_{\ell,k}(\bold x/h-\bold j)\mathbf{f}(\bold jh),
\end{equation}
and
\begin{equation}\label{quasiinterpolatio_curlfreeprojection}
Q^{curl}_h\mathbf{f}(\bold x):=\sum_{\bold j\in \mathbb{Z}^d}\bm{\Psi}^{curl}_{\ell,k}(\bold x/h-\bold j)\mathbf{f}(\bold jh).
\end{equation}
Then we derive  their error estimates.
\begin{theorem}\label{errorestimateforderi2}
Let $\mathbf{f}$ be a vector field whose components $f_s\in \mathcal{S}_p^r(\R^d)$, $s=1,2,\dots, d$ and suppose $\mathbf{f}$ has the Helmoltz-Hodge decomposition $\mathbf{f} = \bold w + \nabla p$ where $\bold w $ is divergence-free and $\nabla p$ is curl-free.  Let $Q^{div}_h\bold f$  and $Q^{curl}_h\bold f$ be defined as above. Then, for $0\leq |\alpha|< \min\{2k, r-2k,2\ell-d-1\}$, we obtain the following simultaneous error estimates: 
\begin{enumerate}
    \item For the divergence-free projection $$\|D^{\alpha}Q^{div}_h\mathbf{f}-D^{\alpha}\mathbf{w}\|_{p,\R^d}=\mathcal{O}(h^{2k-|\alpha|}).$$
    \item For the curl-free projection $$\|D^{\alpha}Q^{curl}_h\mathbf{f}-D^{\alpha}\nabla p\|_{p,\R^d}=\mathcal{O}(h^{2k-|\alpha|}).$$
\end{enumerate}
\end{theorem}
\begin{proof} 
To prove part 1 of the theorem, we note that Equations \eqref{fouriertrans} gives  
 \begin{equation}\label{divfreerojection_fouriertransform}
 \widehat{\bm{\Psi}^{div}_{\ell,k}}(\bm{\omega})=\widehat{\psi_{\ell,k}}(\bm{\omega})\left[\bold I_d-\frac{\bm{\omega}\bm{\omega}^T}{\|\bm{\omega}\|^2}\right].
 \end{equation}
 By fixing a point $\bold x\in \R^d$, we can rewrite  $D^{\alpha}Q^{div}_h\mathbf{f}(\bold x)-D^{\alpha}\mathbf{w}(\bold x)$ using the Fourier transform as
\begin{equation*}
\begin{split}
&D^{\alpha}Q^{div}_h\mathbf{f}(\bold x)-D^{\alpha}\mathbf{w}(\bold x)\\
&=(-i)^{|\alpha|}(2\pi)^{-d/2}\int_{\R^d}\Bigg[h^{-|\alpha|}\sum_{\bold j\in \mathbb{Z}^d}(h\bm{\omega}+2\bold j\pi)^{\alpha}\widehat{\bm{\Psi}^{div}_{\ell,k}}(h\bm{\omega}+2\bold j\pi)\hat{\mathbf{f}}(\bm{\omega})e^{2i\pi \bold j\bold x/h}-\bm{\omega}^{\alpha}\bold I_d\hat{\mathbf{w}}(\bm{\omega})\Bigg]e^{i\bold x\bm{\omega}}d\bm{\omega}.
\end{split}
\end{equation*}
Furthermore,  based on Equation\eqref{divfreerojection_fouriertransform} and Equation \eqref{projectionoperators_Fouriertransform}, we have
\begin{equation*}
\widehat{\bm{\Psi}^{div}_{\ell,k}}(\bm{\omega})\hat{\mathbf{f}}(\bm{\omega}) = 
\widehat{\psi_{\ell,k}}(\bm{\omega})\left[\bold I_d-\frac{\bm{\omega}\bm{\omega}^T}{\|\bm{\omega}\|^2}\right]\hat{\mathbf{f}}(\bm{\omega}) = \widehat{\psi_{\ell,k}}(\bm{\omega})\bold I_d\hat{\mathbf{w}}(\bm{\omega}).
\end{equation*}
This in turn leads to 
\begin{equation*}
\begin{split}
&|D^{\alpha}Q_h^{div}\mathbf{f}(\bold x)-D^{\alpha}\mathbf{w}(\bold x)|\\
&\leq|\bm{\Psi}_{\ell,k}^{div,h}*(D^{\alpha}\mathbf{f})(\bold x)-D^{\alpha}\mathbf{w}(\bold x)|\\
&+h^{-|\alpha|}(2\pi)^{-d/2}\int_{\R^d}\sum_{\bold j\in \mathbb{Z}^d/\{\bold 0\}}|(h\bm{\omega}+2\bold j\pi)^{\alpha}\widehat{{\psi}_{\ell,k}}(h\bm{\omega}+2\bold j\pi) \bold I_d \hat{\mathbf{w}}(\bm{\omega})|d\bm{\omega}\\
&\leq \mathcal{O}(h^{2k})+\mathcal{O}(h^{2k-|\alpha|}),
\end{split}
\end{equation*}
according to  Lemma  \ref{projectionconvolutionerror} and Lemma \ref{polyharmonicsplinepropert}, respectively.
Therefore, we have 
 $\|D^{\alpha}Q^{div}_h\mathbf{f}-D^{\alpha}\mathbf{w}\|_{p,\R^d}=\mathcal{O}(h^{2k-|\alpha|})$.
 
We will now prove part 2 of the theorem. We take 
$\bm{\Psi}^{curl}_{\ell,k}(\bold x) =  {}_{0,0,1}\bm{\Psi}_{\ell,k}(\bold x)$ as an example.
Then Equation \eqref{fouriertrans} gives
 \begin{equation}\label{curlfreerojection_fouriertransform}
 \widehat{\bm{\Psi}^{curl}_{\ell,k}}(\bm{\omega})=\widehat{\psi_{\ell,k}}(\bm{\omega})\frac{\bm{\omega}\bm{\omega}^T}{\|\bm{\omega}\|^2}.
 \end{equation}
Once more, we fix a point $\bold x\in \R^d$ and then rewrite  $D^{\alpha}Q^{curl}_h\mathbf{f}(\bold x)-D^{\alpha}\nabla p(\bold x)$ using the Fourier transform as
\begin{equation*}
\begin{split}
&D^{\alpha}Q^{curl}_h\mathbf{f}(\bold x)-D^{\alpha}\nabla p(\bold x)\\
&=(-i)^{|\alpha|}(2\pi)^{-d/2}\int_{\R^d}\Bigg[h^{-|\alpha|}\sum_{\bold j\in \mathbb{Z}^d}(h\bm{\omega}+2\bold j\pi)^{\alpha}\widehat{\bm{\Psi}^{div}_{\ell,k}}(h\bm{\omega}+2\bold j\pi)\hat{\mathbf{f}}(\bm{\omega})e^{2i\pi \bold j\bold x/h}-\bm{\omega}^{\alpha}\bold I_d(\bm{\omega}\hat{p})(\bm{\omega})\Bigg]e^{i\bold x\bm{\omega}}d\bm{\omega}.
\end{split}
\end{equation*}
Using Equations \eqref{curlfreerojection_fouriertransform}, \eqref{projectionoperators_Fouriertransform}, and the properties of $\mathbf{f}$, we have
\begin{equation*}
\widehat{\bm{\Psi}^{div}_{\ell,k}}(\bm{\omega})\hat{\mathbf{f}}(\bm{\omega}) = 
\widehat{\psi_{\ell,k}}(\bm{\omega})\frac{\bm{\omega}\bm{\omega}^T}{\|\bm{\omega}\|^2}\hat{\mathbf{f}}(\bm{\omega}) = \widehat{\psi_{\ell,k}}(\bm{\omega})\bold I_d(\bm{\omega}\hat{p})(\bm{\omega}).
\end{equation*}
Similarly, based on  Lemma \ref{projectionconvolutionerror} and Lemma \ref{polyharmonicsplinepropert}, we can derive the error estimate
$\|D^{\alpha}Q^{curl}_h\mathbf{f}-D^{\alpha}\nabla p\|_{p,\R^d}=\mathcal{O}(h^{2k-|\alpha|})$.
\end{proof}

\subsubsection{\textit{Quasi-interpolation defined over a bounded domain $\Omega$}}\label{sec:divfreequasioncompactdomain}
In the available literature on quasi-interpolation there is a well-known boundary problem. Namely, quasi-interpolation provides poor approximations at the boundary if we apply it directly to the data sampled over a bounded domain. 
Recently, authors in Reference \cite{Gaoetal2} proposed an approach to mitigate this issue that uses the Boolean sum of interpolation (near the boundary) and quasi-interpolation (in the interior) of the bounded domain. The resulting approximation still performs well at the boundary without any extension (as in \cite{ChenandSuter}, for instance). Based on this success, we adopt the approach described by \cite{Gaoetal2} here and throughout the remainder of the paper.

 We now describe our method for quasi-interpolation defined over a bounded domain. Let $V$ be a bounded domain satisfying $V \subset \Omega$ and ${\rm dist}(\partial \Omega, \partial V) \ge c$ for a fixed $0<c<1$. Here ${\rm dist}(\partial \Omega, \partial V)$ denotes the Hausdorff distance between $\partial \Omega$ and $\partial V$. Let $\mathbb{A}$ be a finite subset of $\mathbb{Z}^d$ and $\mathbb{B}$ be a finite subset of $\mathbb{A}$. Let further $\{(\bold jh,\bold f(\bold jh)\}_{\bold j\in \mathbb{A}}$ be sampling data over the domain $\Omega$ with $\{(\bold jh,\bold f(\bold jh)\}_{\bold j\in \mathbb{B}}\subset \Omega/V$. Then, based on the data $\{(\bold jh,\bold f(\bold jh)\}_{\bold j\in \mathbb{B}}\subset \Omega/V$, we construct an \emph{interpolant} $I\bold f$  interpolating $\bold f$ at these sampling centers following any approach in \cite{fsw}, \cite{fw2}, \cite{fw}, such that it can be used for computing the Helmholtz--Hodge decomposition. For example, $I\bold f$ is the generalized interpolant of Equation (13) in \cite{fsw}.  
Furthermore, based on $I\bold f$ and $\bold f$, we  construct another vector-valued function $\bold g(\bold x):=\bold f(\bold x)-I\bold f(\bold x)$, $\bold x\in \Omega$. In contrast to the divergence-free quasi-interpolation scheme proposed in \cite{Gaoetal2}, $\bold g$ and $I\bold f$ are no longer divergence-free vector-valued functions. Thus, to compute the quasi-interpolant projection of $\bold f$ onto the space of divergence-free functions, we must also compute the divergence-free part of $I\bold f$, which we denote by $I^{div}\bold f$ (see the unnumbered equation after (20) in \cite{fsw}). Moreover,  we choose  $I\bold f$ appropriately such that $\bold g=(g_1,g_2,\dots, g_d)$ satisfies  $g_s\in \mathcal{S}_p^r(\R^d)$, $s=1,2,\dots, d$. We denote by $Q_{h,H}^{div}\bold g$  the divergence-free quasi-interpolation projection operator $Q_{h,H}^{div}$ applied directly to the sampling data $\{(\bold jh,\bold g(\bold jh)\}_{\bold j\in \mathbb{A}}$, that is, \begin{equation}\label{quasiinterpolationonboundeddomain}
Q_{h,H}^{div}\bold g(\bold x)=h^d\sum_{\bold j\in \mathbb{A}}\bm{\Psi}_{\ell,k}^{div,H}(\bold x-\bold jh)\bold g(\bold jh)
\end{equation} with $$\bm{\Psi}_{\ell,k}^{div,H}(\bold x)=H^{-d}\bm{\Psi}_{\ell,k}^{div}(\bold x/H).$$ Finally, by coupling  $I^{div}\bold f$ together with $Q^{div}_{h,H}\bold g$, we  construct an {\it ansatz} $IQ^{div}_{h,H}\bold f$  as
\begin{equation}\label{combined}
IQ^{div}_{h,H}\bold f(\bold x )=I^{div}\bold f(\bold x )+Q^{div}_{h,H}\bold g(\bold x ),\ \bold x \in \Omega.
\end{equation}
It is straightforward to verify that $IQ^{div}_{h,H}\bold f$ is a divergence-free vector-valued function, since both $I^{div}\bold f$ and $Q^{div}_{h,H}\bold g$ are divergence-free functions. We now derive the approximation errors to the Leray projection, $\mathcal{P}_L\bold f$, of function $\bold f$ and its derivatives.

In what follows we assume $1/2\leq \mu \leq p$. Then, based on Jensen's inequality, we have  \begin{equation*}
\begin{split}
\|D^{\alpha}IQ^{div}_{h,H}\bold f-D^{\alpha} \mathcal{P}_{L} \bold f\|_{\mu,\Omega}&=\left(\int_{\Omega}|D^{\alpha}IQ^{div}_{h,H}\bold f(\bold x)-D^{\alpha}\mathcal{P}_{L} \bold f(\bold x)|^\mu d\bold x\right)^{1/\mu}\\
&=\left(\int_{V}|D^{\alpha}IQ^{div}_{h,H}\bold f(\bold x)-D^{\alpha} \mathcal{P}_{L}  \bold f(\bold x)|^{\mu}d\bold x+\int_{\Omega/V}|D^{\alpha}IQ^{div}_{h,H}\bold f(\bold x)-D^{\alpha} \mathcal{P}_{L} \bold f(\bold x)|^\mu d\bold x\right)^{1/\mu}\\
&\leq \|D^{\alpha}IQ^{div}_{h,H}\bold f-D^{\alpha}\mathcal{P}_{L} \bold f\|_{\mu,V}+\|D^{\alpha}IQ^{div}_{h,H}\bold f-D^{\alpha} \mathcal{P}_{L} \bold f\|_{\mu,\Omega/V}.
\end{split}
 \end{equation*} We note that $\|\bold f\|_{\mu, V}=\sum_{s=1}^d\|f_s\|_{\mu, V}$ with $\|f_s\|_{\mu, V}=(\int_V|f_s(\bold x)|^{\mu}d\bold x)^{1/{\mu}}$ for $f_s\in \mathcal{S}_p^r(\R^d)$, $s=1,2,\dots d$.  Moreover, according to Equation \eqref{quasiinterpolationonboundeddomain} and Equation \eqref{combined}, we have
\begin{equation*}
\|D^{\alpha}IQ^{div}_{h,H}\bold f-D^{\alpha}\mathcal{P}_{L} \bold f\|_{\mu,V}=\|D^{\alpha}Q^{div}_{h,H}\bold g-D^{\alpha} (\mathcal{P}_{L} \bold f-I^{div}\bold f)\|_{\mu,V},
\end{equation*}
and
 \begin{equation*}
\|D^{\alpha}IQ^{div}_{h,H}\bold f-D^{\alpha}\mathcal{P}_{L} \bold f\|_{\mu,\Omega/V}=\|D^{\alpha}Q_{h,H}^{div}(\bold f-I\bold f)-D^{\alpha}(\mathcal{P}_{L} \bold f-I^{div}\bold f)\|_{\mu,\Omega/V}.
 \end{equation*}
Furthermore, observing that
\begin{equation*}
\begin{split}
&\|D^{\alpha}Q^{div}_{h,H}(\bold f-I\bold f)-D^{\alpha}(\mathcal{P}_{L}\bold f-I^{div}\bold f)\|_{\mu,\Omega/V}\\
&\leq \|D^{\alpha}Q^{div}_{h,H}(\bold f-I\bold f)\|_{\mu,\Omega/V}+\|D^{\alpha}(\mathcal{P}_{L}\bold f-I^{div}\bold f)\|_{\mu,\Omega/V}\\
&\leq H^{-|\alpha|-d}\|{}_{1,1,0}\bm{\Psi}_{\ell,k}\|_{\infty}\|\bold f-I\bold f\|_{\mu,\Omega/V}+\|D^{\alpha}\mathcal{P}_{L}\bold f-D^{\alpha}I^{div}\bold f\|_{\mu,\Omega/V},
 \end{split}
 \end{equation*}
 it suffices to derive bounds of the quasi-interpolation error $\|D^{\alpha}Q^{div}_{h,H}\bold g-D^{\alpha}(\mathcal{P}_{L} \bold f-I^{div}\bold f)\|_{\mu,V}$ and  the interpolation errors $\|\bold f-I\bold f\|_{\mu,\Omega/V}$ and $\|D^{\alpha}\mathcal{P}_{L}\bold f-D^{\alpha}I^{div}\bold f\|_{\mu,\Omega/V}$, where $0\leq|\alpha|\leq r$.  In addition, since radial basis function interpolation methods for the Helmholtz--Hodge decomposition have been proposed and studied in several papers, we adopt their results. In particular, according to Lemma 5.2 in \cite{fw}, for  the generalized interpolant $I\bold f$   we have (for $\bold f \in \bold H^p(\Omega/V)$)
 $$\|\bold f-I\bold f\|_{\mu,\Omega/V}  = \mathcal O(h^{p-\mu}), \quad 0\leq \mu \leq p. $$
 On the other hand, according to Theorem 5.5 in \cite{fw}, we have 
 $$\|\mathcal{P}_{L}\bold f-I^{div}\bold f\|_{\mu,\Omega/V}  = \mathcal O(h^{p-\mu}), \quad 1/2\leq \mu\leq p. $$
Next, we assume that the interpolant provides an approximation order of at least $n$ over the domain $\Omega/V$ for some positive integer $n$ satisfying $\max\{1,|\alpha|\}<n\leq p-\mu$. This in turn implies that $$\|D^{\alpha}(\mathcal{P}_L\bold f-I^{div}\bold f)\|_{\mu,\Omega/V}=\mathcal{O}(h^{n-|\alpha|}),$$ and consequently $$\|D^{\alpha}Q^{div}_{h,H}(\bold f-I\bold f)-D^{\alpha}(\mathcal{P}_L\bold f-I^{div}\bold f)\|_{\mu,\Omega/V}\leq \mathcal{O}(h^nH^{-|\alpha|-d})+\mathcal{O}(h^{n-|\alpha|}).$$

Setting $\bold g_L = \mathcal{P}_L\bold f-I^{div}\bold f$, it remains to derive a bound of  the quasi-interpolation error $\|D^{\alpha}Q^{div}_{h,H}\bold g-D^{\alpha}\bold g_L \|_{\mu,V}$.
 We adopt the technique proposed in \cite{Gaoetal} by viewing quasi-interpolation as a two-step procedure: first convolution and then discretization of the convolution. Thus, we first introduce two convolution operators $$\mathcal{C}^{div}_H(\bold g)(\bold x):=\int_{\R^d}\bm{\Psi}_{\ell,k}^{div,H}(\bold x-\bold t)\bold g(\bold t)d\bold t$$ and $$\mathcal{C}^{div}_{\Omega,H}(\bold g)(\bold x):=\int_{\Omega}\bm{\Psi}_{\ell,k}^{div,H}(\bold x-\bold t)\bold g(\bold t)d\bold t.$$ We remind readers that $\mathcal{C}^{\mu}_H$ is defined for $L^\mu$ functions on $\R^d$ and $\mathcal{C}^{div}_{\Omega,H}$ for those functions when restricted on $\Omega$. Moreover, using the Minkowski inequality, we have
\begin{equation}\label{Minkowski}
\|D^{\alpha}Q^{div}_{h,H}\bold g-D^{\alpha} \bold g_L \|_{\mu, V} \leq  \|D^{\alpha}\mathcal{C}^{div}_{\Omega,H}(\bold g)-D^{\alpha} \bold g_L \|_{\mu, V}+\|D^{\alpha}Q^{div}_{h,H}\bold g-D^{\alpha}\mathcal{C}^{div}_{\Omega,H}(\bold g)\|_{\mu,V}.\end{equation}
Our task now is to derive bounds of these two parts on the right-hand side of the above inequality.

\begin{lemma}\label{convolutionoveraboundeddomain}
 Let $V$ be a bounded domain satisfying $V \subset \Omega$ and ${\rm dist}(\partial \Omega, \partial V) \ge c$ for a fixed $0<c<1$. Here ${\rm dist}(\partial \Omega, \partial V)$ denotes the Hausdorff distance between $\partial \Omega$ and $\partial V$.   Let $\bold f$ be a vector field whose components $f_s\in \mathcal{S}_p^r(\R^d)$, $s=1,2,\dots d$. Let $I\bold f$ be the generalized interpolant of $\bold f$ and $I^{div}\bold f$ be the interpolant Leray projection such that $\bold g:=\bold f-I\bold f$ and $\bold g_L := \mathcal{P}_L \bold f-I^{div} \bold f$ are functions whose components are in $\mathcal{S}_p^r(\R^d)$. Assume further that the divergence-free interpolant $I^{div}\bold f$ provides an approximation order of at least $\mathcal{O}(h^n)$ in the domain $\Omega/V$, for some positive integer $n$ satisfying $\max\{1,|\alpha|\}<n\leq p-\mu$. Let $Q^{div}_{h,H}\bold g$ be defined as in Equation \eqref{quasiinterpolationonboundeddomain}. Then, for $1/2 \leq \mu \leq p$, we have the inequality
 \begin{equation}\label{convolutionerror}
 \|D^{\alpha}\mathcal{C}^{div}_{\Omega,H}(\bold g)-D^{\alpha}\bold g_L \|_{\mu, V} \leq \mathcal{O}(H^{2k-2|\alpha|-\varepsilon})+\mathcal{O}(h^n/H^{d+|\alpha|})
\end{equation}
holds true for any small positive $\varepsilon$ with $0<\varepsilon<1$ and any $0\leq |\alpha|\leq \min\{k, r-2k,2\ell-d-1\}$.
 \end{lemma}
 \begin{proof}
 Since $$D^{\alpha}\mathcal{C}^{div}_{H}(\bold g)=\mathcal{C}^{div}_{H}(D^{\alpha}\bold g),$$ we have the equality $$\left\|D^{\alpha}\bold g_L -D^{\alpha}\mathcal{C}^{div}_{H}(\bold g)\right\|_{\mu, V}=\mathcal{O}(H^{2k})$$ holds true for any $0\leq |\alpha|\leq \min\{2k, r-2k\}$ according to Lemma \ref{projectionconvolutionerror}. Thus it remains to derive the  bound for $\left\|D^{\alpha}\mathcal{C}^{div}_{H}(\bold g)-D^{\alpha}\mathcal{C}^{div}_{\Omega,H}(\bold g) \right\|_{\mu, V}$.

Using integration by parts, we have
\begin{equation*}
\begin{split}
D^{\alpha}\mathcal{C}^{div}_{\Omega,H}( \bold g)(\bold x)&=\int_{\Omega}(-1)^{\ell}q_{d,\ell, k}(\tilde{\Delta})\Delta\bold I_dD^{\alpha}\phi_{\ell+1}^H(\bold x-\bold t) \mathcal{P}_L\bold g(\bold t)d\bold t\\
&+[(-1)^{\ell}q_{d,\ell, k}(\tilde{\Delta})\nabla D^{\alpha}\phi_{\ell+1}^H(\bold x-\bold t)(\mathcal{P}_L\bold g(\bold t)\cdot \bold 1)]_{\bold t \in \partial \Omega}\\
&+[(-1)^{\ell}q_{d,\ell, k}(\tilde{\Delta})\nabla D^{\alpha}\phi_{\ell+1}^H(\bold x-\bold t)\times(\bold 1 \times \mathcal{P}^\perp_L\bold g(\bold t))]_{\bold t \in \partial \Omega}
\end{split}
\end{equation*} for any sufficiently smooth vector field $\bold g$ with Helmholtz--Hodge decomposition $\bold g = \mathcal{P}_L\bold g + \mathcal{P}^\perp_L\bold g$. Here we remind the readers that the differential operator is taken with respect to the $\bold x$ variable. Furthermore, since $I\bold f$ provides an approximation of at least order $\mathcal{O}(h^n)$ in the domain $\Omega/V$,  we have
$$\left|[(-1)^{\ell}q_{d,\ell, k}(\tilde{\Delta})\nabla D^{\alpha}\phi_{\ell+1}^H(\bold x-\bold t)(\mathcal{P}_L\bold g(\bold t)\cdot \bold 1)]_{\bold t \in \partial \Omega}\right|= \mathcal{O}(h^n/H^{d+|\alpha|})$$
and 
$$\left|[(-1)^{\ell}q_{d,\ell, k}(\tilde{\Delta})\nabla D^{\alpha}\phi_{\ell+1}^H(\bold x-\bold t)\times(\bold 1 \times \mathcal{P}^\perp_L\bold g(\bold t))]_{\bold t \in \partial \Omega}\right|= \mathcal{O}(h^n/H^{d+|\alpha|}).$$

Consequently, we have
$$D^{\alpha}\mathcal{C}^{div}_{\Omega,H}( \bold g)(\bold x)=\int_{\Omega}(-1)^{\ell}q_{d,\ell,k}(\tilde{\Delta})\Delta\bold I_dD^{\alpha}\phi_{\ell+1}^H(\bold x-\bold t)\mathcal{P}_L\bold g(\bold t)d\bold t+\mathcal{O}(h^n/H^{d+|\alpha|}).$$

 Therefore, fixing an $\bold x\in V$, we have  \begin{equation*}
  \begin{split}
 & \left|D^{\alpha}\mathcal{C}^{div}_{H}(\bold g)(\bold x)-D^{\alpha}\mathcal{C}^{div}_{\Omega,H}(\bold g)(\bold x) \right|\\
  &\leq \left|\int_{\R^d \setminus \Omega}\bold I_dD^{\alpha}\mathbf{\psi}_{\ell,k}^H(\bold x-\bold t)\mathcal{P}_L\bold g(\bold t)d\bold t\right|+\mathcal{O}(h^n/H^{d+|\alpha|})\\
  &\leq H^{-|\alpha|}\left|\int_{|\bold y|\geq c/H}\bold I_d\mathbf{\psi}_{\ell,k}^{(\alpha)}(\bold y)\mathcal{P}_L\bold g(\bold x-H\bold y)d\bold y\right|+\mathcal{O}(h^n/H^{d+|\alpha|})\\
  &\leq H^{-|\alpha|}\|\mathcal{P}_L\bold g\|_{\infty}\int_{|\bold y|\geq c/H}|\bold y|^{-2k+|\alpha|+\varepsilon}|\bold y|^{2k-|\alpha|-\varepsilon} |\bold I_d\mathbf{\psi}_{\ell,k}^{(\alpha)}(\bold y)|d\bold y+\mathcal{O}(h^n/H^{d+|\alpha|})\\
  &\leq H^{2k-2|\alpha|-\varepsilon}c^{-2k+|\alpha|+\varepsilon}\|\mathcal{P}_L\bold g\|_{\infty}\int_{|\bold y|\geq c/H}|\bold y|^{2k-|\alpha|-\varepsilon} |\bold I_d\mathbf{\psi}_{\ell,k}^{(\alpha)}(\bold y)|d\bold y+\mathcal{O}(h^n/H^{d+|\alpha|})\\
  &\leq CH^{2k-2|\alpha|-\varepsilon}\int_{\R^d}|\bold y|^{2k-|\alpha|-\varepsilon} |\bold I_d\mathbf{\psi}_{\ell,k}^{(\alpha)}(\bold y)|d\bold y+\mathcal{O}(h^n/H^{d+|\alpha|}).
  \end{split}
  \end{equation*}
  Moreover, based on Equality \eqref{decayofderivatives} we have $\int_{\R^d}|\bold y|^{2k-|\alpha|-\varepsilon} |\bold I_d\mathbf{\psi}_{\ell,k}^{(\alpha)}(\bold y)|d\bold y<+\infty$, which in turn yields
  $$\left\|D^{\alpha}\mathcal{C}_{H}(\bold g)-D^{\alpha}\mathcal{C}_{\Omega,H}(\bold g) \right\|_{p,V}\leq \mathcal{O}(H^{2k-2|\alpha|-\varepsilon})+\mathcal{O}(h^n/H^{d+|\alpha|})$$ for any $0\leq |\alpha|\leq \min\{k, r-2k,2\ell-d-1\}$.
Thus the lemma holds.
 \end{proof}
In addition, since $Q^{div}_{h,H}\bold g$ is a discretization of the integral $\mathcal{C}^{div}_{\Omega,H}( \bold g)$  using the rectangular rule with a step $h$, we have
 \begin{equation}\label{discretizationerror}
 |D^{\alpha}Q^{div}_{h,H}\bold g(\bold x)-D^{\alpha}\mathcal{C}^{div}_{\Omega,H}( \bold g)(\bold x)|=\mathcal{O}(hH^{-d-|\alpha|}).
 \end{equation}
This entails the following lemma.
 \begin{lemma}\label{errorestimateonboundeddomain}
  Let $V$,  $I\bold f$,  $I^{div}\bold f$, $\bold g$, $\bold g_L$,  and $Q^{div}_{h,H}\bold g$ satisfy the conditions in Lemma \ref{convolutionoveraboundeddomain}.  Then, for any vector field $\mathbf{f}$  whose components satisfy $f_s\in \mathcal{S}_p^r(\R^d)$, $s=1,2,\dots d$, and for $1/2 \leq \mu \leq p$, we have the inequality
 \begin{equation}\label{erroronboundednew}
 \|D^{\alpha}Q^{div}_{h,H}\bold g-D^{\alpha}\bold g_L \|_{\mu, V} \leq \mathcal{O}(H^{2k-2|\alpha|-\varepsilon})+\mathcal{O}(hH^{-d-|\alpha|})
\end{equation}
holds true for any small positive $\varepsilon$ with $0<\varepsilon<1$ and any $0\leq |\alpha|< \min\{k, r-2k,2\ell-d-1\}$.
 \end{lemma}
\begin{proof}
 Inequality \eqref{erroronboundednew}  follows directly from Inequalities \eqref{Minkowski}- \eqref{discretizationerror}.
\end{proof}

The main result of this section is obtained by collecting the above results into a theorem.
\begin{theorem}\label{fianlerrorestimateonboundeddomain}
Let $V$, $\bold f$, $I\bold f$, $I^{div}\bold f$, and $Q^{div}_{h,H}\bold g$  satisfy the conditions in Lemma \ref{convolutionoveraboundeddomain}, and suppose $\mathbf{f}$ has the Helmholtz--Hodge decomposition $\mathbf{f} = \mathbf{w} + \nabla p$. Then we have
 \begin{equation}\label{erroronboundedop}
 \|D^{\alpha}IQ^{div}_{h,H}\bold f-D^{\alpha}\bold w\|_{\mu,\Omega}\leq  \mathcal{O}(H^{2k-2|\alpha|-\varepsilon})+\mathcal{O}(hH^{-d-|\alpha|}).
\end{equation}
 In particular, we can determine an optimal choice of the scale parameter $H$ to be $$H=\mathcal{O}(h^{1/(2k+d-|\alpha|-\varepsilon)}),$$ such that we have an optimal approximation error estimate $$\|D^{\alpha}IQ^{div}_{h,H}\bold f-D^{\alpha}\bold w\|_{\mu,\Omega}\leq \mathcal{O}(h^{(2k-2|\alpha|-\varepsilon)/(2k+d-|\alpha|-\varepsilon)}).$$
 \end{theorem}
 \begin{remark}
 For brevity and ease of presentation, the above discussion has focused on computing the divergence-free component of the Helmholtz-Hodge decomposition. However, it should be noted that similar results can be derived for the curl-free component as well using the ansatz 
 \begin{equation}\label{combined2}
IQ^{curl}_{h,H}\bold f(\bold x )=I^{curl}\bold f(\bold x )+Q^{curl}_{h,H}\bold g(\bold x ),\ \bold x \in \Omega,
\end{equation}
where $I^{curl}\bold f(\bold x )$ and $Q^{curl}_{h,H}\bold g(\bold x )$ are the curl-free counterparts of $I\bold f(\bold x )$ and  $Q_{h,H}\bold g(\bold x )$, respectively.
\end{remark}
 
 Up to now, we have constructed quasi-interpolation schemes for the Helmholtz--Hodge decomposition and derived  corresponding  convergence analysis for vector fields defined both in the whole space and over a bounded domain, respectively. In the next section, we shall provide some simulations.

\section{\textit{Numerical Simulations}}
In this section we will provide numerical confirmation of the above theoretical results. In particular, we will simulate functions and derivative approximations on $\R^d$ as well as compute the quasi-interpolant Leray projection of a vector field on the bounded domain $\Omega = [0,1]\times[0,1].$
As noted in \cite{Gaoetal2}, the construction of the kernel using the fundamental solution \eqref{fundamentalsolution} will require the use of a continuous extension at the origin (and possibly other points in the domain). To avoid a repetitive discussion of this issue, we assume the continuous extension is handled appropriately throughout the numerical experiments.

\subsection{Approximations on $\R^d$}
In order to simulate approximations on $\R^d$ we set $d = 2$, for simplicity, and we construct a vector field whose components are known to be divergence-free and curl-free. In particular, we choose 
\[
\frac{\partial^{\alpha} {\bf f }}{\partial x_1^\alpha} = 
\frac{\partial^{\alpha} }{\partial x_1^\alpha} \left( \bold d(\bold x) + \bold c(\bold x)\right)
 \] 
with divergence-free component 
 \[
\bold d(\bold x) = 
\begin{bmatrix}
    \sin(2\pi x_2) \sin^2(\pi x_1),
   -\sin(2\pi x_1) \sin^2(\pi x_2)\end{bmatrix}^T,
 \] 
and  curl-free component 
  \[
\bold c(\bold x ) = 
 \begin{bmatrix}
    \pi \sin(\pi x) \sin(\pi y),
   -\pi \cos(\pi x) \cos(\pi y)\end{bmatrix}^T,
 \] 
 for $\alpha = 0,1$.
The function $\bold f$ is sampled over a large domain $[0,12]\times[0,12]$ using the grid spacings $h = 12/(18+6i)$ for $i = 0,\dots, 14$ and we compute the error on a fine uniform mesh of $100\times100$ points in the interior evaluation domain $[5.5,6.5]\times[5.5,6.5]$. We run the code using the kernels corresponding to a fixed $\ell = k = 2$, i.e., $\boldsymbol{\Psi}^{div}_{2,2}(\bold x)$  and 
$\boldsymbol{\Psi}^{curl}_{2,2}(\bold x)$, 
 and we measure the root mean squared error of the approximation with respect to the divergence-free component
\[ error_{div} = \frac{\partial^{\alpha} }{\partial x_1^\alpha}Q^{div}_h \bold f - \frac{\partial^{\alpha} {\bf d }}{\partial x_1^\alpha},\]
the curl-free component
\[ error_{curl} =\frac{\partial^{\alpha} }{\partial x_1^\alpha}Q^{curl}_h \bold f - \frac{\partial^{\alpha} {\bf c }}{\partial x_1^\alpha},\]
and the full vector field
\[ error_{full} =\left( \frac{\partial^{\alpha} }{\partial x_1^\alpha}Q^{div}_h \bold f + \frac{\partial^{\alpha} }{\partial x_1^\alpha}Q^{curl}_h \bold f\right) - \frac{\partial^{\alpha} {\bf f }}{\partial x_1^\alpha}.\]
\begin{figure}
  \centering
{\includegraphics[scale=.43] {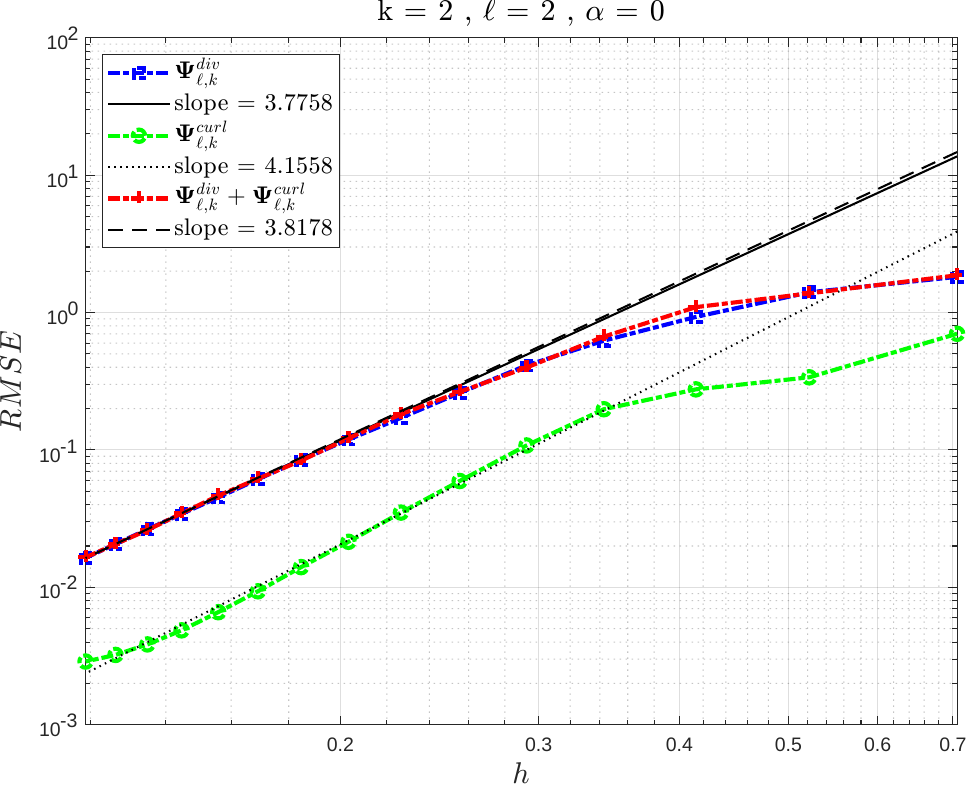}}
{\includegraphics[scale=.45] {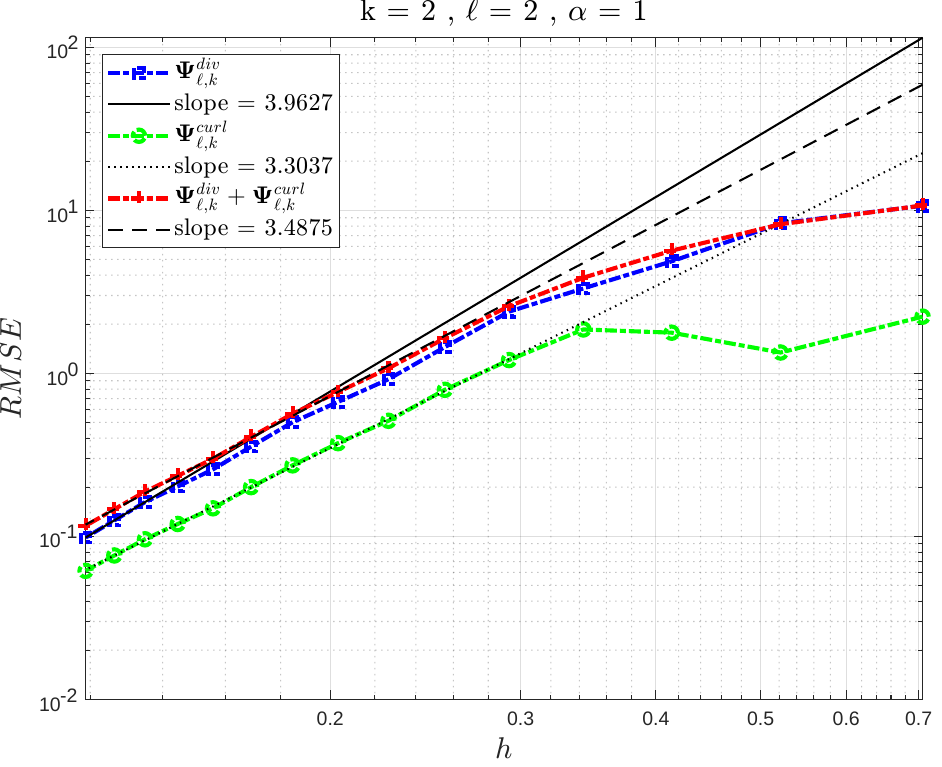}}
  \caption{Errors and convergence rates on $\R^d$ with $\alpha = 0$ (left) and $\alpha = 1$ (right).}
  \label{fig:1}
\end{figure}

The results are shown in Figure \ref{fig:1}. Each of the slopes is computed using the last 10 observed errors, indicating the expected convergence rates ($\mathcal{O}(h^4)$ for $\alpha = 0$ and $\mathcal{O}(h^3)$ for $\alpha = 1$) are obtained asymptotically.

\subsection{Approximation on bounded domains}

As in the previous numerical experiments on $\R^d$, we require a vector field whose curl-free and divergence-free components are known. In this case, we choose ${\bf f} = \bold d(\bold x) + \bold c(\bold x)$
 with
 \[
\bold d(\bold x) = 
\begin{bmatrix}
    \sin(8\pi x_2) \sin^2(4\pi x_1),
   -\sin(8\pi x_1) \sin^2(4\pi x_2)\end{bmatrix}^T,
 \] 
and 
  \[
\bold c(\bold x ) = 
 \begin{bmatrix}
    2\pi \sin(2\pi x) \sin(2\pi y),
   -2\pi \cos(2\pi x) \cos(2\pi y)\end{bmatrix}^T.
 \] 
 Note that for this choice of $\bold f$, the boundary conditions of Remark \ref{rem:boundaryconditions} are satisfied.

 We run the code using the kernels $\boldsymbol{\Psi}^{div}_{2,2}(\bold x) $ and $\boldsymbol{\Psi}^{div}_{3,3}(\bold x) $
for the quasi-interpolant and the $C^{8}$ Mat\'{e}rn kernel as the choice of scalar RBF for the interpolant.
In each ease we use $H= C h^{1/(2k + d -\epsilon)}$ setting $\epsilon = 10^{-3}$ and $C = \mathcal{O}(10^{-2})$.
The function ${\bf f}$ is then sampled on the domain $\Omega = [0,1]\times[0,1]$ using $h = 1/(5+10i)$, $i =1, \dots, 9 $ and we evaluate the error on a fine uniform mesh of $100\times100$ points. The domain 
$V \subset \Omega$ of Section \ref{sec:divfreequasioncompactdomain} is selected to be $V = [0.1,0.9]\times[0.1,0.9]$. We measure the root mean squared error of the quasi-interpolant Leray projection 
 \[error_{Leray} = \mathcal{P}_LIQ_{h,H} \bold f - {\bf d }.\] 

 The results provided  in Figure \ref{fig:2} show the expected convergence rates. That is, computing the slopes using the last five available errors we find the approximation is $\mathcal{O}(h^{2/3})$ when $k = \ell = 2$ and $\mathcal{O}(h^{3/4})$ when $k = \ell = 3$, as predicted. Once more we see that the convergence rates are obtained asymptotically. In Figure \ref{fig:3} we show some typical results of using the quasi-interpolant for Helmholtz--Hodge decomposition.  From left to right the columns show the observed data, the computed Leray projection, the computed curl-free component, and the sum of the two computed components, respectively. We can find that our quasi-interpolant provides excellent approximations to   the full field as well as its corresponding divergence-free share and curl-free share.

\begin{figure}
  \centering
{\includegraphics[scale=0.5] {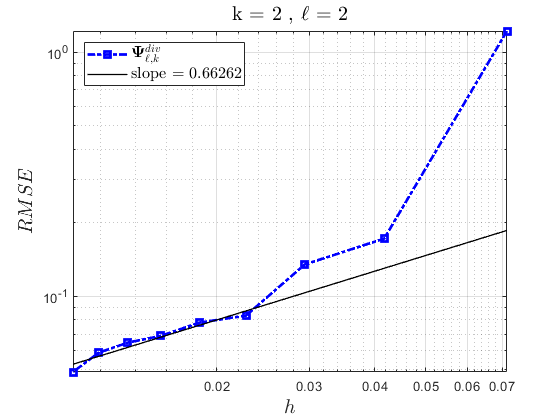 }}{\includegraphics[scale=0.5] {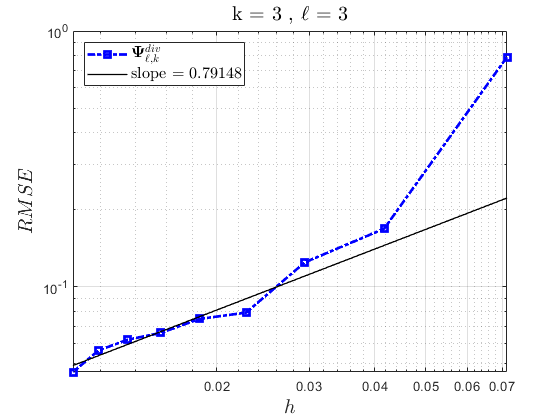 }}

  \caption{Errors and convergence rates on $[0,1]^2$ with $\alpha = 0$ using $\boldsymbol{\Psi}^{div}_{2,2}(\bold x) $ (left) and $\boldsymbol{\Psi}^{div}_{3,3}(\bold x)$ (right).}
  \label{fig:2}
\end{figure}

\begin{figure}
  \centering
{\includegraphics[scale=0.8] {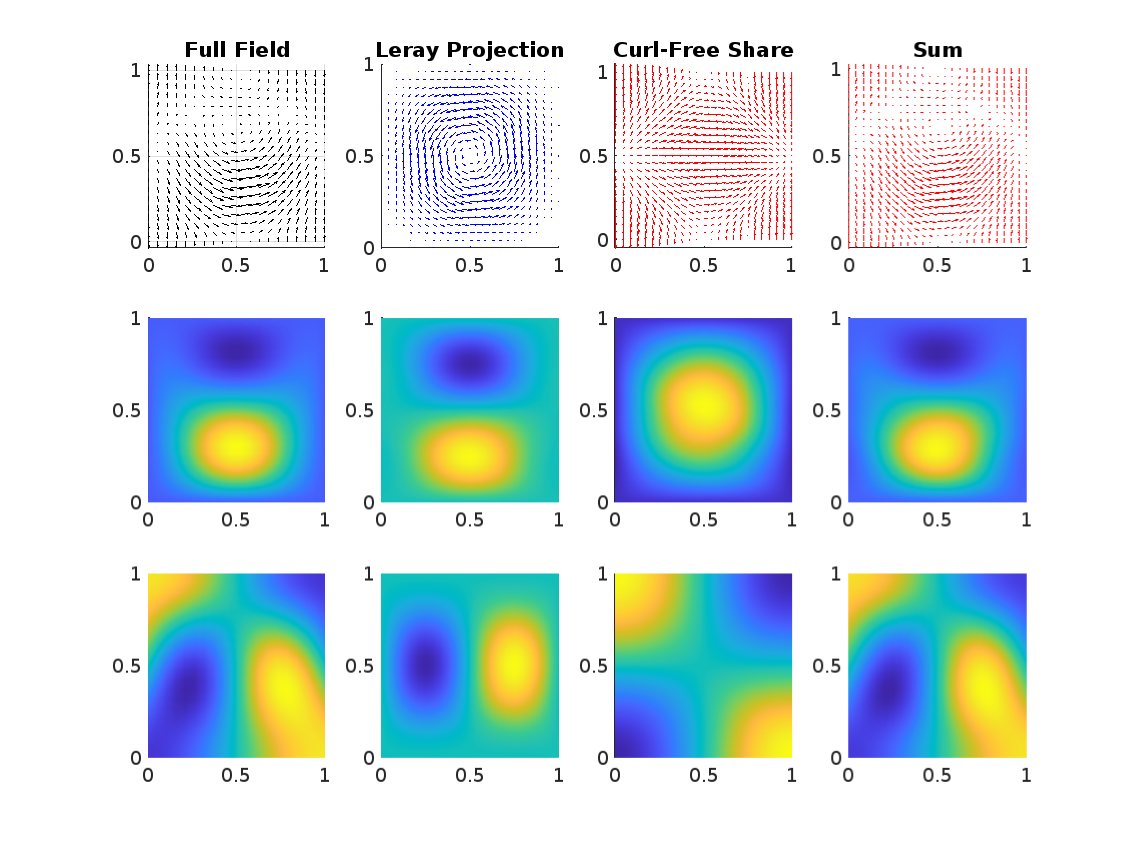 }}

  \caption{Quasi-interpolation for the Helmholtz--Hodge decomposition on a bounded domain.}
  \label{fig:3}
\end{figure}

 \section{\textit{Conclusions and Discussions}}
The paper provides a  quasi-interpolation technique for approximating a smooth vector field and its individual Helmholtz-Hodge components. The technique naturally decomposes the reconstructed field into  divergence-free and curl-free parts, which approximates the  corresponding counterparts of the approximated vector field. Simultaneous error estimates are derived both in the whole space and over a bounded domain. Numerical simulations provided at the end of the paper show that our quasi-interpolation technique is both computationally efficient and numerically stable for  vector field reconstruction as well as  Helmholtz-Hodge decomposition.
Future works will focus on extending the technique to vector field approximation on manifold.


\section*{Funding}
This work is supported in part by funds from the National Science Foundation (NSF: \# 1636933 and \# 1920920),  the National Science Foundation of China (NSFC: 12271002), and Yong Outstanding Talents of Universities of Anhui Province (No.2024AH020019), and by the NSF RTG grant DMS-2136228.



\begin{thebibliography}{80}\label{ref:ref}






 \bibitem{Amodei} L. Amodei, M. N. Benbourhim, A vector spline quasi-interpolation. In {\sl Wavelets, Images, and Surface Fitting (Chamonix-MontBlanc, 1993)}, A K Peters, Wellesley, MA, 1994, 1--10.
  \bibitem{Atteia} M. Atteia, M. N. Benbourhim, P. G. Casanova, Quasi-interpolant elastic manifolds. In {\sl Reporte de Investigation IIMAS}. UNAM 3 (25) Mexico (1993).
\bibitem{BeatsonPowell} R. Beatson,  M. Powell, Univariate multiquadric approximation: Quasi-interpolation to scattered data, Constr. Approx. 8 (1992) 275--288.

 \bibitem{Benbourhim} M. N. Benbourhim, P. G. Casanova, Generalized variational quasi-interpolants in $(H^m(\Omega))^n$,  Bol. Soc. Mat. Mexicana 3 (1997).
\bibitem{Benbourhim1} M. N. Benbourhim, A. Bouhamidi, Meshless pseudo-polyharmonic divergence-free and curl-free vector fields approximation, SIAM J. Numer. Anal. 42 (2010) 1218--1245.
\bibitem{Buhmann} M. Buhmann, Convergence of univariate quasi-interpolation using multiquadrics, IMA J. Numer. Anal. 8 (1988) 365--383.
\bibitem{Buhmann1} M. Buhmann, On quasi-interpolation with radial basis functions, J. Approx. Theory  72 (1993) 103--130.
\bibitem{BuhmannandDyn}  M. Buhmann, N. Dyn, D. Levin,   On quasi-interpolation by radial basis functions with scattered centres, Constr. Approx. 11 (1995) 239--254.
\bibitem{ChenandSuter} F. Chen, D. Suter, Div-curl vector quasi-interpolation on a finite domain, Math. Comput. Model.  30 (1999) 179--204.
\bibitem{Cheney}  E. Cheney, W.   Light, Y. Xu,   On kernels and approximation orders, In {\sl Approximation Theory (Memphis, TN)}, Lecture Notes in Pure and Applied Mathematics, Dekker, New York, 138 (1992) 227--242.
\bibitem{CharlesChuiandDiamond}  C. K. Chui, H. Diamond,   A natural formulation of quasi-interpolation by multivariate splines,  Proc. Amer. Math. Soc. 99 (1987) 643--646.
\bibitem{CharlesChuiandDiamond1}  C. K. Chui, H. Diamond,   A characterization of multivariate quasi-interpolation formulas and its applications, Numer. Math. 57 (1990) 105--121.
\bibitem{CharlesChuiandLai} C. K. Chui, M. J. Lai, A multivariate analog of Marsden's identity and a quasi-interpolation scheme, Constr. Approx. 3 (1987) 111--122.
\bibitem{chorin}
A.J.~Chorin, Numerical solution of the Navier-Stokes equations,
Math. Comput., 22 (1968), 745--762.
\bibitem{cucker-1} Cucker, F., Smale, S., On the mathematical foundations of learning, Bull. Amer. Math. Soc. 39 (2001) 1-49.
\bibitem{cucker-2} Cucker, F., Zhou, D. X.,  Learning Theory: An Approximation Theory Viewpoint (Cambridge Monographs on Applied and Computational Mathematics), Cambridge University Press, 2007.
\bibitem{Deriaz} E. Deriaz and V. Perrier, Orthogonal Helmholtz decomposition in arbitrary dimension using divergence-free and curl-free wavelets. Appl. Comput. Harmon. Anal.26 (2009) 249--269.
\bibitem{DoduandRabut} F. Dodu, C. Rabut, Irrotational or divergence-free interpolation, Numer. Math. 98 (2004) 477--498.
\bibitem{Drake} K. Drake, G. Wright, A stable algorithm for divergence-free radial basis functions in the flat limit, arXiv.2001.04557v2.
\bibitem{DynandJackRon} N. Dyn, I. R. H. Jackson,  D. Levin, A. Ron, On multivariate approximation by integer translates of a basis function, Israel J. Math. 78 (1992)  95--130.
\bibitem{Freeden} W. Freeden and T. Gervens, Vector spherical spline interpolation-basic theory and computational aspects, Math. Methods Appl. Sci. 16 (1993) 151--183.
\bibitem{Fuselier} E. Fuselier, Sobolev-type approximation rates for divergence-free and curl-free RBF interpolants, Math. Comput. 77 (2008) 1407--1423.
 \bibitem{Fuselier1} E. Fuselier,  F. Narcowich, J. Ward, G. Wright,  Error and stability estimates for surface-divergence free RBF interpolants on the sphere, Math. Comput. 78 (2009) 2157--2186.
 \bibitem{fsw}E. Fuselier, V.~Shankar, and G.~Wright,
A high-order radial basis function (RBF) Leray projection method for the solution of the incompressible unsteady Stokes equations, Computers and Fluids, 128 (2016) 41-52.
 \bibitem{fw2}E. Fuselier and G.~Wright,
Stability and error estimates for vector field interpolation and decomposition on the sphere with RBF's, SIAM J. Numer. Anal. 47 (2009), 3213–3239
 \bibitem{fw}E. Fuselier and G.~Wright,
A radial basis function method for computing Helmholtz--Hodge Decompositions, IMA J. Numer. Anal. 37 (2017) , 774–797
\bibitem{GaoandWu4} W. W. Gao, Z. M.  Wu,  Constructing radial kernels with higher-order moment conditions, Adv. Comput. Math.  43 (2017) 1355--1375.
\bibitem{Gaoetal} W. W. Gao, G. E. Fasshauer, X. P. Sun, X. Zhou, Optimality and regularization properties of quasi-interpolation: deterministic  and stochastic approaches, SIAM J. Numer. Anal. 58 (2020) 2059--2078.
\bibitem{Gaoandsun} W. W. Gao, X. P. Sun, Z. M. Wu, and X. Zhou, Multivariate Monte Carlo approximation based on scattered data, SIAM J. Sci. Comput. 42 (2020) A2262-A2280.
\bibitem{Gaoetal2} W. W. Gao, G. E. Fasshauer, and N. Fisher, Divergence-free quasi-interpolation, Appl. Comp. Harmon, Anal. 60 (2022) 471-488.
\bibitem{Grohs} P. Grohs, Quasi-interpolation in Riemannian manifolds, IMA J. Numer. Anal. 33 (2013) 849--874.
 \bibitem{Grohs1} P. Grohs, M. Sprecher, T. Yu, Scattered manifold-valued data approximation, Numer. Math. 135 (2017) 987--1010.
\bibitem{gms}
J.L.~Guermond, P.~Minev , and J.~Shen,
{\em An overview of projection methods for incompressible flow},
Comput. Methods App. Mech. and Engrg., 195 (2009), 6011--6045.
 \bibitem{GuoandZhou} X. Guo, D. X. Zhou, An empirical feature-based learning algorithm producing sparse approximations, Appl. Comput. Harmon. Anal. 32 (2012) 389-400.
\bibitem{JiaandLei} R. Q. Jia, J. J. Lei, A new version of Strang--Fix conditions, J. Approx. Theory 74 (1993) 221--225.
\bibitem{LeijiaandCheney} J. J. Lei, R. Q. Jia, E. Cheney,  Approximation from shift-invariant spaces by integral operators, SIAM J. Math. Anal. 28 (1997) 481--498.
\bibitem{Wendland1} C. Keim, H. Wendland, A high-order, analytically divergence-free approximation method for the time-dependent Stokes problem, SIAM J. Numer. Anal. 54 (2016) 1288--1312.
\bibitem{Lanzaraandmazya} F. Lanzara, V. Maz'ya,  G. Schmidt,  Approximate approximations from scattered data, J. Approx. Theory 145 (2007) 141--170.
\bibitem{Lanzaraandmazya1} F. Lanzara, V. Maz'ya,  G. Schmidt,   Approximation of solutions to multidimensional parabolic equations by approximate approximations,  Appl. Comput. Harmon. Anal.  36 (2014) 167--182.
 \bibitem{Lanzaraandmazya2} F. Lanzara, V. Maz'ya,  G. Schmidt,   Fast cubature of volume potentials over rectangular domains by approximate approximations,  Appl. Comput. Harmon. Anal.  41 (2016) 749--767.
\bibitem{Narcowich} F. Narcowich, J. Ward, Generalized Hermite interpolation via matrix-valued conditionally positive definite functions, Math. Comput. 63 (1994) 661--687.
 \bibitem{Narcowich1} F. Narcowich, J. Ward, G. Wright, Divergence-free RBFs on surfaces,  J. Fourier Anal. Appl. 13 (2007) 643--663.
 \bibitem{pp} K. Polthier and E. Pruess, Identifying vector field singularities using a discrete Hodge decomposition, Visualization and Mathematics III (H.C. Hedge and K. Polthier eds.) Springer, Berlin.
 \bibitem{CR} C. Rabut, An introduction to Schoenberg's approximation, Comput. Math. Appl. 24 (1991) 139--175.
\bibitem{Rabut0}  C. Rabut, How to build quasi-interpolants: applications to polyharmonic B-splines. In {\sl Curves and Surfaces}, P. J. Laurent, A. Le M\'ehaut\'e, and L. L. Schumaker (eds.). Academic Press, New York, (1991) 391--402.
\bibitem{Rabut}  C. Rabut, Elementary $m$-harmonic cardinal B-splines, Numer. Algor. 2 (1992) 39--62.
  \bibitem{Rabut1}  C. Rabut,  High level $m$-harmonic cardinal B-splines, Numer. Algor. 2 (1992) 63--84.

\bibitem{Schabackandwu} R. Schaback, Z. M. Wu, Construction techniques for highly accurate quasi-interpolation operators, J. Approx. Theory 91 (1997) 320--331.
\bibitem{Schoenberg} I. J. Schoenberg, Contributions to the problem of approximation of equidistant data by analytic functions, Quart. Appl. Math. 4 (1946) 45-99 and 112--141.
\bibitem{Wendland3} D. Schr\"ader, H. Wendland, A high-order, analytically divergence-free discretization method for Darcy's problem, Math. Comput. 80 (2011) 263--277.
\bibitem{Schwarz} G. Schwarz, {\em Hodge Decomposition-A Method for Solving Boundary Value Problems}, Lecture Notes in Mathematics, 1607, Springer, Berlin, vii-155.
\bibitem{Smaleandzhou} S. Smale, D. X. Zhou, Shannon sampling II: Connections to learning theory, Appl. Comput. Harmon. Anal. 19 (2005) 285-302.
\bibitem{Stein} E. Stein, G. Weiss, Introduction to Fourier Analysis on Euclidean Spaces, Princeton Univ. Press, Princeton, 1971.

\bibitem{Strang} G. Strang, G. Fix, A Fourier analysis of the finite-element method, In {\sl Constructive Aspects of Functional Analysis}, G. Geymonat (ed.) (C.I.M.E., Rome, 1973) 793--840.
\bibitem{SunandWu} H. Sun, Q. Wu, A note on application of integral operator in learning theory, Appl. Comput. Harmon. Anal.  26 (2009) 416-421.
\bibitem{temam}R.~Temam,{ Sur l'approximation de la solution des \'equations de Navier-Stokes pare la m\'ethod de pas fractionnaries, II}. Arch. Ration. Mech. Anal., 33 (1969), 337--385.
\bibitem{VenelandBeatson} R. Vennell, R. Beatson, A divergence-free spatial interpolator for large sparse velocity data sets, J. Geophys. Res. 114 (2009) 10--24.
  \bibitem{Wendland5} H. Wendland,  Divergence-free kernel methods for approximating the Stokes problem, SIAM J. Numer. Anal. 47 (2009) 3158--3179.
\bibitem{ZMRS0} Z. M. Wu, R. Schaback, Local error estimates for radial basis function interpolation of scattered data, IMA J. Numer.  Anal. 13 (1993) 13-27. 
\bibitem{ZMRS} Z. M. Wu, R. Schaback, Shape preserving properties and convergence of univariate multiquadric quasi-interpolation, Acta. Math. Appl. Sin. 10 (1994) 441--446.
\bibitem{GeneralizedWu} Z.M. Wu, J. P. Liu, Generalized Strang--Fix condition for scattered data quasi-interpolation, Adv. Comput. Math. 23(2005)  201--214.


%


%
%
%
%
%
%
%
%
%
%
%


\end{thebibliography}

\end{document}